\documentclass[11pt]{amsart}

\usepackage{amsthm}
\usepackage{amsxtra}
\usepackage{amssymb}
\usepackage{bbm}
\usepackage{amsfonts}
\usepackage{textcomp}
\usepackage{lmodern}

\usepackage{color}

\newtheorem{theorem}{Theorem}
\numberwithin{theorem}{section}

\newtheorem{lemma}[theorem]{Lemma}

\newtheorem{corollary}[theorem]{Corollary}

{\theoremstyle{definition}
\newtheorem{remark}[theorem]{Remark}
}

\numberwithin{equation}{section}
\numberwithin{figure}{section}

\renewcommand{\Re}{\mathop{\rm Re}\nolimits}

\newcommand{\R}{\mathbb R}
\newcommand{\Rbarplus}{{\ts\overline{\R_+}\ts}}
\newcommand{\C}{\mathbb{C}}
\newcommand{\N}{\mathbb{N}}
\newcommand{\Z}{\mathbb{Z}}

\newcommand{\la}{\langle}
\newcommand{\ra}{\rangle}
\newcommand{\hf}{{\frac 12}}

\newcommand{\hatL}{{\widehat L^1}}

\newcommand{\Cos}{\mathbf C}
\newcommand{\Sin}{\mathbf S}
\newcommand{\V}{\mathbf{V}}
\newcommand{\Lam}{\mathbf \Lambda}
\renewcommand{\L}{\mathcal{L}}

\newcommand{\bfs}{\mathbf{s}}

\newcommand{\one}{{\mathbf{1}}}

\newcommand{\ts}{\hspace{.06em}}

\DeclareMathOperator{\trace}{Tr}
\DeclareMathOperator{\comp}{c}

\DeclareMathOperator{\supp}{supp}
\DeclareMathOperator{\diam}{diam}

\title [Trace of the wave group and regularity of potentials]{On the trace of the wave group and regularity of potentials}
\author[H. F. Smith]{Hart F. Smith}
\address{Department of Mathematics, University of Washington, Seattle, WA 98195-4350, USA}
\email{hfsmith@uw.edu}

\keywords{Wave trace, Schr\"odinger operators}
  
\subjclass[2010]{58J45, 58J50, 46E35}

\begin{document}

\begin{abstract}
For the wave equation $\partial_t^2-\Delta+V$ on $\mathbb{R}^d$ with compactly supported, real valued potential $V$, we establish a sharp relation between Sobolev regularity of $V$ and the existence of finite order expansions as $t\rightarrow 0$ for the relative trace of the wave group.
\end{abstract}

\maketitle

\section{Introduction and statement of results}
Suppose that $V$ is a real valued, bounded measurable function of compact support on $\R^d$.
Define the wave group $U_V$ by $U_V(f,g)=(u,\partial_t u)$, where $u$ is the solution to the initial value problem for the following wave equation
$$
\bigl(\partial_t^2-\Delta+V\bigr)u(t,x)=0\,,\quad 
(u,\partial_t u)\bigr|_{t=0}=(f,g)\in H^1\times L^2(\R^d).
$$
Let $U_V(s)$ denote the map $(f,g)\rightarrow (u,\partial_t u)|_{t=s}$.

If $V\in C_{\comp}^\infty(\R^d)$, it is well known that for $\phi\in C_{\comp}^\infty(\R)$ the operator
$$
\la\phi,U_V-U_0\ra=\int \phi(t)\bigl(U_V(t)-U_0(t)\bigr)\,dt
$$
is trace class; see for example \cite{DyZw}, \cite{Me1}, \cite{Me2}, \cite{SaBZw}. The map $\phi\rightarrow \trace\la\phi,U_V-U_0\ra$ defines an even distribution on $\R$, denoted by $\trace(U_V-U_0)$. For $V\in C_{\comp}^\infty(\R^d)$ this distribution is smooth on $\R\backslash\{0\}$, and admits an asymptotic expansion in $t$ for $t$ near $0$.
If $d$ is odd, the expansion takes the form
$$
\trace(U_V-U_0)=\sum_{j=1}^{(d-1)/2}w_j(V) D^{d-1-2j}\delta(t)+\sum_{j=(d+1)/2}^\infty w_j(V)\,|t|^{2j-d},
$$
and if $d$ is even it takes the form
$$
\trace(U_V-U_0)=\sum_{j=1}^{(d-2)/2} w_j(V)D^{d-1-2j}\text{p.v.}(1/t)+\sum_{j=d/2}^\infty w_j(V)\,t^{2j-d}.
$$
The expansion holds in the sense that the difference between the left side and the finite sum to $j=m$ is a continuous function on $\R$ that is $\mathcal{O}(|t|^{2m+2-d})$ for $t$ near $0$, provided $m\ge d/2$.

In this paper we consider the existence of similar expansions to finite order when $V$ is of limited regularity, and show that for real valued, compactly supported potentials $V$, existence of such an expansion to order $j=m+2$ is equivalent to the Sobolev regularity condition $V\in H^m(\R^d)$. 
Our proof shows that $\la \phi, U_V-U_0\ra$ is of trace-class for $V\in L^\infty_{\comp}(\R^d)$, but if $d\ge 4$ then the proof of the expansion in $t$ for the trace relies on $L^1$ bounds on the Fourier transform $\widehat V$ of $V$, and thus for $d\ge 4$ we need to place the stronger a priori assumption that $\widehat V\in L^1(\R^d)$ when proving bounds on various remainder terms.

We define the space $\hatL(\R^d)\subset L^\infty(\R^d)$ with norm
$$
\|V\|_{\hatL}=(2\pi)^{-d}\int|\widehat V(\xi)|\,d\xi,
$$
and let $\hatL_{\comp}\subset L^\infty_{\comp}$ denote the space of $V\in \hatL$ with compact support.

We define the space $X_d$ according to whether $d\le 3$ or not,
\begin{equation*}
X_d=\begin{cases} L^\infty_{\comp}(\R^d),&\;\;d\le 3,\\ 
\hatL_{\comp}(\R^d),&\;\;d\ge 4.\rule{0pt}{18pt}
\end{cases}
\end{equation*}
Finally, let
$$
\|V\|_{X_d\cap H^m}=\|V\|_{X_d}+\|V\|_{H^m},
$$
and note that $H^s_{\comp}(\R^d)\subset X_d$ if $s>d/2$.

\begin{theorem}\label{thm:main}
Assume that $V\in L^\infty_{\comp}(\R^d)$. If $\phi\in C_{\comp}^\infty(\R)$, then the operator on $H^1(\R^d)\times L^2(\R^d)$ defined by 
$$
(f,g)\rightarrow \int\phi(t)\Bigl(U_V(t)(f,g)-U_0(t)(f,g)\Bigr)\,dt
$$
is trace-class, and its trace defines an even distribution acting on $\phi$, which we denote by $\trace\la \phi,U_V-U_0\ra$.

If $V\in X_d$, and $\phi\in C_{\comp}^\infty(\R)$ is an even function, there are distributions $\nu_d$ and $\mu_{d,V}$ on $\R$ such that
$$
\trace\la \phi,U_V-U_0\ra=\nu_d(\phi)\int V(x)\,dx+\mu_{d,V}(\phi)\,
$$
where $\mu_{d,V}$ has the following form, with $\alpha_V(t)\in C(\overline{\R_+})$,
\begin{equation*}
\mu_{d,V}(\phi)=
\begin{cases}
\displaystyle\int_0^\infty t^{4-d}\alpha_V(t)\,\phi(t)\,dt,& 1\le d\le 4,\\
\displaystyle\int_0^\infty \alpha_V(t)\,\partial_t(t^{-1}\partial_t)^{\frac{d-5}2}\phi(t)\,dt,& d\ge 5\;\text{odd},\rule{0pt}{20pt}\\
\displaystyle\int_0^\infty \alpha_V(t)\,(t^{-1}\partial_t)^{\frac{d-4}2}\phi(t)\,dt,& d\ge 6\;\text{even}.\rule{0pt}{20pt}
\end{cases}
\end{equation*}
If $V\in X_d\cap H^m(\R^d)$, $m$ an integer, then
$\alpha_V(t)\in C^{2m}(\Rbarplus)$, and its Taylor expansion at $0$ is of the form
\begin{equation}\label{eqn:traceW}
\alpha_V(t)=\sum_{j=0}^m a_{j+2}(V)\ts t^{2j}+o(t^{2m})\,,
\end{equation}
where $a_{j+2}(V)$ is a dimension dependent multiple of the Schr\"odinger heat invariant $c_{j+2}(V)$.

Conversely, if $V\in X_d$ is real valued, and there are constants $b_j\in \C$ and $C<\infty$ so that
\begin{equation}\label{eqn:traceW'}
\biggl|\alpha_V(t)-\sum_{j=0}^{m-1} b_j \ts t^{2j}\biggr|\le C\,t^{2m}\quad\text{for}\quad 0\le t\le 1,
\end{equation}
then $V\in H^{m}(\R^d)$, and hence \eqref{eqn:traceW} holds.
\end{theorem}

We note that the $a_j(V)$ in \eqref{eqn:traceW} differ by a constant from the $w_j(V)$ in the expansion presented for smooth potentials.
The form of $\nu_d$ is given in Theorem \ref{thm:traceW1}, and agrees with the term $j=1$ in the expansion for smooth $V$ above. It can also be written in a form similar to that for $\mu_{d,V}$.

The proof of Theorem \ref{thm:main} establishes pointwise bounds on $\alpha_{V}$ and its derivatives. In particular the following holds,
$$
|\alpha_{V}(t)|\le C_d\,\|V\|_{L^2}^2\cosh\Bigl(t\ts \|V\|_{X_d}^{1/2}\Bigr).
$$

An analogous result for the heat kernel of $-\Delta+V$ on Euclidean space was established by the author and Zworski in \cite{SmZw}. There it was proven for $V\in L^\infty_{\comp}(\R^d)$ in all dimensions $d$. The same result was subsequently established by the author in \cite{Sm} for the heat kernel on complete Riemannian manifolds under mild geometric assumptions. The study of heat traces has an extensive history, going back to Kac, Berger and McKean-Singer. For detailed discussions of the heat trace for Schr\"odinger operators we refer to Ba\~nuelos-S\`a Baretto \cite{BaSaB}, Colin de Verdi\`ere \cite{CdV}, Hitrik-Polterovich \cite{HiPo}, and Gilkey \cite{Gilkey}. The coefficients $c_j(V)$ in the expansion of the relative heat trace, known as heat invariants, are a key tool in the proof of compactness of isospectral families of potentials; see Br\"uning \cite{Br} and Donnelly \cite{Don}.

There appears to be far fewer treatments of the trace of the wave kernel for $-\Delta+V$, even though it is a more fundamental operator in the sense that the trace expansion for the heat kernel follows easily from that for the wave kernel, but not vice versa; see Remark \ref{rem:heat} below. The wave invariants $w_j(V)$ are dimension dependent multiples of the heat invariants $c_j(V)$, as noted by S\'a Baretto-Zworski \cite{SaBZw}, thus on a compact manifold they do not give new spectral invariants. In Euclidean space of odd dimension, on the other hand, the wave-trace restricted to $t\ne 0$ is related to the resonances of $-\Delta+V$ by the trace formula of Melrose; see \cite[Theorem 3.53]{DyZw}. As a consequence, if $V_1$ and $V_2$ are bounded real potentials with the same set of resonances and eigenvalues, counted with the appropriate multiplicity, then
$$
t^d \trace(U_{V_1}-U_0)=t^d \trace(U_{V_2}-U_0).
$$

As observed by Hislop-Wolf \cite{HiWo}, in odd dimensions this shows that if the $V_j$ are smooth potentials, then $V_j$ have the same heat coefficients, $c_k(V_1)=c_k(V_2)$, if $k\ge (d+1)/2$.
This in turn is used in \cite{HiWo} to establish a compactness theorem for iso-resonant families of $C_{\comp}^\infty$ potentials with support in a common ball, when $d=1$ or $3$. 

A different proof of this result was given by Christiansen, who proved equality of $c_k(V_j)$ for $k\ge (d+1)/2$ by using the Birman-Krein formula to relate the trace of the heat kernel to the poles of the scattering phase. See \cite[Theorem 1.2]{SmZw} for a discussion of her proof in odd dimensions. This method was extended by Christiansen in \cite{TC1} to prove iso-resonant compactness and related results for $C_{\comp}^\infty$ potentials in even dimensions.

One of the motivations for this paper was a remark by the authors of \cite{HiWo} that existence of the expansion of Theorem \ref{thm:main} would extend certain results of \cite{HiWo} to potentials of Sobolev regularity. Precisely, if $d$ is odd then
Theorem \ref{thm:main} can be combined with the Melrose trace formula to prove equality of the heat coefficients $c_k(V)$ for $(d+1)/2\le k\le m+2$, assuming $V\in X_d\cap H^m(\R^d)$ is real and $m\ge (d-3)/2$. Christiansen's method combined with the heat trace expansion of \cite{SmZw} also implies this result.
In dimensions $d=1,3$, this implies $\|V_1\|_{L^2}=\|V_2\|_{L^2}$, and one can bootstrap to obtain the following inverse result.

\begin{theorem}
Suppose that $d=1,3$ and that $V_1, V_2\in L^\infty_{\comp}(\R^d)$ are iso-resonant. If $V_1\in H^m(\R^d)$, where $m\ge 0$ is an integer, then $V_2\in H^m(\R^d)$, and there is a polynomial bounded function $F_m$ such that
\begin{equation}\label{eqn:Sobbound}
\|V_2\|_{H^m(\R^d)}\le F_m\bigl(\|V_1\|_{H^m(\R^d)}\bigr).
\end{equation}
\end{theorem}

Given equality of the heat coefficients for $V_1$ and $V_2$,
the bound \eqref{eqn:Sobbound} follows from the papers \cite{Br} and \cite{Don}, where it was shown that for $1\le d\le 3$, one can bound
$$
\|V\|_{H^m}\le P_{m,d}\bigl(c_2(V),\ldots,c_{m+2}(V)\bigr),
$$
for a dimension dependent polynomial $P_{m,d}$. This is obtained by induction from a bound of the form
\begin{equation*}
\||D|^m V\|_{L^2}\le C_{m,d}\, c_{m+2}(V)+P_{m,d}\bigl(\|V\|_{H^{m-1}}\bigr).
\end{equation*}
The papers \cite{Br} and \cite{Don} work on compact manifolds, but the estimates hold globally on $\R^d$.

Consequently, the family of $L^\infty_{\comp}$ real potentials that are iso-resonant to a given potential in $H^m(\R^d)$ forms a bounded subset of $H^m(\R^d)$, if $d=1,3$. That this family is closed is shown in \cite{HiWo}. If one restricts attention to potentials with support in fixed ball of finite radius, then by Rellich's Lemma such potentials are a compact subset of $H^s(\R^d)$ if $s<m$.

\section{Preliminary reductions}

We prove Theorem \ref{thm:main} using the iterative construction of $U_V$. Define
$$
\Cos(t)=\cos\bigl(t|D|\bigr),\qquad \Sin(t)=\frac{\sin\bigl(t|D|\bigr)}{|D|}.
$$
Then the free wave group is given by
$$
U_0(t)=\begin{bmatrix} \Cos(t) & \Sin(t) \\ \Delta\Sin(t) & \Cos(t)\end{bmatrix}.
$$

Let $\Lam$ denote the following operator on $C([-T,T];H^1\times L^2)$, $T<\infty$,
$$
\Lam (F,G)(t,\cdot)=\int_0^t U_0(t-s)(F,G)(s,\cdot)\,ds.
$$
Then 
$$
U_V-U_0=
\sum_{k=1}^\infty (-\Lam \V)^kU_0, \qquad \V=\begin{bmatrix} 0 & 0 \\ V & 0 \end{bmatrix},
$$
where the sum converges absolutely in the $H^1\times L^2\rightarrow C(H^1\times L^2)$ operator norm, for each $T<\infty$.

We recall some basic properties of trace-class operators. For a compact operator $A$ on a Hilbert space with principal values $\sigma(A)$, one defines the Schatten-$p$ norms of $A$ for $1\le p\le\infty$ by
$$
\|A\|_{\L^p}=\biggl(\sum_{j=1}^\infty \sigma_j(A)^p\biggr)^{1/p}.
$$
The $\L^\infty$ norm is the operator norm, and $\L^1$ is the trace-class norm. As a corollary of \cite[Theorem 1]{KyFan} of Ky Fan, we note the following H\"older type estimate,
\begin{equation}\label{eqn:Holder}
\|A_1A_2\cdots A_n\|_{\L^1}\le \prod_{j=1}^n \|A_j\|_{\L^{p_j}}\quad\text{if}\quad\sum_{j=1}^n p_j^{-1}=1,\;\; 1\le p_j\le\infty.
\end{equation}
When $p_j=\infty$ we may replace the $\L^\infty$ norm by the operator norm, and assume only that $A_j$ is a bounded operator, not necessarily compact. This holds since \cite[Theorem 1]{KyFan} follows from uniform estimates over finite rank operators.

\begin{lemma}\label{lem:tracesum}
If $V\in L^\infty_{\comp}(\R^d)$, then for $\phi\in C^\infty_{\comp}(\R)$ the sum
\begin{equation}\label{eqn:tracesum}
\sum_{k=1}^\infty (-1)^k\int_{-\infty}^\infty \phi(t)\bigl((\Lam \V)^kU_0\bigr)(t)\,dt
\,\equiv\, \sum_{k=1}^\infty (-1)^k\la \phi,W_{k,V}\ra
\end{equation}
converges absolutely to $\la \phi,U_V-U_0\ra$ in the trace-class norm on $H^1\times L^2(\R^d)$.
\end{lemma}
\begin{proof}
Let $R=1+\diam(\supp(V))$, and given $t>0$ let $K$ be a cube of sidelength $\ell(K)=2(R+t)$ that contains all points within distance $t$ of $\supp(V)$.
Then for $|s|<t$ the Schwartz kernel of $V\Sin(s)$ is supported entirely within $K\times K$, so the Schatten-$p$  norm of $V\Sin(s)$ on $L^2(\R^d)$ is the same as that of its restriction to $L^2(K)$.
We use the Fourier series basis for $L^2(K)$ with characters $\xi\in \Xi_K$,
\begin{equation*}
e_\xi=|K|^{-1/2}e^{i\langle x,\xi\rangle}\,,\qquad\xi\in\Xi_K=2\pi\, \ell(K)^{-1}\Z^d.
\end{equation*}
In this basis, we have
$$
V\Sin(s)e_\xi=\frac{\sin(s|\xi|)}{|\xi|}Ve_\xi.
$$
Thus $V\Sin(s)$ factors into $V$ times a diagonal operator, and we have
\begin{equation}\label{eqn:Schattenbound}
\begin{split}
\bigl\|V\Sin(s)\|_{\mathcal{L}^{d+1}}&\le\|V\|_{L^\infty}\Biggl(\,\sum_{\xi\in\Sigma_K}\biggl|\frac{\sin(s|\xi|)}{|\xi|}\biggr|^{d+1}\Biggr)^{1/(d+1)}\\
&\le C_d\|V\|_{L^\infty}\Bigl(t(R+t)^d\Bigr)^{1/(d+1)}.
\end{split}
\end{equation}
Here we use the bound
$$
\biggl|\frac{\sin(s|\xi|)}{|\xi|}\biggr|\le \frac{2s}{1+s|\xi|}\,,
$$
and we use the fact that $\ell(K)^{-1}s\le 1$ to estimate the sum by an integral.

We first show that the lower right corner $\la \phi,W_{k,V}^{2,2}\ra$ of the $2\times 2$-block matrix in \eqref{eqn:tracesum} converges in the trace class norm on $L^2$. All terms in the sum for this term are even in $t$, so we may assume that $\phi(t)$ is even in $t$.
Let $d^k\! s=ds_1\cdots ds_k$.
Then $\la \phi,W_{k,V}^{2,2}\ra$ can be written as
\begin{multline*}
2\int_{\substack{s_1+\cdots s_k < t\\ s_1,\ldots,s_k>0}}\phi(t)\Cos(t-(s_1+\cdots +s_k))V\Sin(s_k)V\cdots V \Sin(s_1)\,d^k\! s\,dt
\\
=2\int_{\R_+^{k+1}}\phi(s_1+\cdots+s_{k+1})\Cos(s_{k+1})V\Sin(s_k)V\cdots V \Sin(s_1)\,d^{k+1}\!s.
\end{multline*}
First consider the case $k\ge d+1$. For $0\le s\le t$ we have the bound
$$
\|\one_{\supp(V)}\Sin(s)\one_{\supp(V)}\|_{L^2(K)\rightarrow L^2(K)}\le (C_dR)\wedge t,
$$
with $R=1+\diam(\supp(V))$ as above.
From \eqref{eqn:Holder} and \eqref{eqn:Schattenbound}, we conclude that, with $\mathcal{L}^1$ denoting the trace-class norm,
\begin{multline*}
\biggl\|\int_{\substack{s_1+\cdots s_k < t\\ s_1,\ldots,s_k>0}}\Cos(t-(s_1+\cdots +s_k))V\Sin(s_k)V\cdots V \Sin(s_1)\,d^k\!s\biggr\|_{\mathcal{L}^1}\\
\le
\frac{C_d\|V\|_{L^\infty}^k(R+t)^dt^{k+1}\bigl((C_dR)^{k-d-1}\wedge t^{k-d-1}\bigr)}{k!}.
\end{multline*}
We conclude that if $k\ge d+1$ then
$W_{k,V}^{2,2}$ is trace-class, with $\L^1$ norm given by a continuous function $\beta_{k,V}(t)$ satisfying 
\begin{equation}\label{eqn:akbound}
|\beta_{k,V}(t)|\le 
\biggl( \frac{C_d\ts R^d \ts t^{2k-d}\ts \|V\|_{L^\infty}^k}{k!}\biggr)\wedge\biggl(\biggl(\frac t R\biggr)^{d+1}\frac{\bigl(C_d\|V\|_{L^\infty}R\, t\bigr)^k}{k!}\,\biggr).
\end{equation}
The sum of $W_{k,V}^{2,2}$ over $k\ge d+1$ thus converges absolutely in the trace class.

For later use we note that
\begin{equation}\label{eqn:aksumbound}
\sum_{k=d+1}^\infty |\beta_{k,V}(t)|\le 
\biggl(\frac t R\biggr)^{d+1}e^{C_dR\ts t\ts\|V\|_{L^\infty}}.
\end{equation}
It remains to show that $\la\phi,W_{k,V}^{2,2}\ra$ is trace class for $k\le d$. Suppose that $\supp(\phi)\subset [-T,T]$, and let $K$ be the cube of sidelength $2(T+R)$, which we identify with the torus by imposing periodic boundary conditions on $\Delta$. In bounding the $\L^1$ norm of $\la\phi,W_{k,V}^{2,2}\ra$, we may then assume we are working with the wave equation on $K$.

We write
$$
\Sin(s)=(1-\partial_s^2)\Sin(s)\,(1-\Delta)^{-1}.
$$
Integration by parts in $s_1$ shows that $\frac 12\la\phi,W_{k,V}^{2,2}\ra$ equals the following,
\begin{multline*}
\int_{\R_+^{k+1}}(\phi-\phi'')(s_1+\cdots +s_{k+1})\Cos(s_{k+1})V\Sin(s_k)\cdots V \Sin(s_1)\,d^{k+1}\!s\cdot(1-\Delta)^{-1}\\
+
\int_{\R_+^k}\phi(s_2+\cdots +s_{k+1})\Cos(s_{k+1})V\Sin(s_k)\cdots V \Sin(s_2)\, d^k\!s \cdot V(1-\Delta)^{-1}.
\end{multline*}
We repeat this on each term, stopping when the total number of $\Sin(s_j)$ factors plus twice the number of $(1-\Delta)^{-1}$ factors in a term is at least $d+1$. Thus, we stop after at most $d+1-k$ steps. 

If $2k\le d+1$, there will arise terms of the form
$\widehat{\phi^{(2m)}}\bigl(\sqrt{-\Delta}\bigr)Q$ with $2m\le d+1-2k$, and $Q$ consisting of a mixed product of $k$ factors of $V$ and $m+k$ factors of $(1-\Delta)^{-1}$, hence
$$
\|Q\|_{\L^p}\le C_d\|V\|_{L^\infty}^k\bigl(T+R\bigr)^{d/p},\quad \text{if}\;\; p=(d+1)/(2m+2k).
$$
For this value of $p$, we can bound
$$
\|\widehat{\phi^{(2m)}}\bigl(\sqrt{-\Delta}\bigr)\|_{\L^{p'}}\le C_d\bigl(T+R\bigr)^{d/p'}\bigl(\,\|\phi^{(2m)}\|_{L^1}+\|\phi^{(d+1-2k)}\|_{L^1}\bigr).
$$

The terms that involve an integral over $s$ are bounded as for $k\ge d+1$, and together we conclude that, if $\supp(\phi)\subset [-T,T]$, and $1\le k\le d$,
\begin{equation}\label{eqn:trphiWk}
\bigl\|\la\phi,W_{k,V}\ra\bigr\|_{\mathcal{L}^1}\le C_d\,(T+R)^{d+1}\,\|V\|_{L^\infty}^k\sum_{n=0}^{d+1-k}\int_\R |\phi^{(n)}(t)|\,dt.
\end{equation}
This concludes the proof of $\L^1$ summability over $k$ of $\la\phi,W_{k,V}^{2,2}\ra$.

A similar proof shows that each of the following is absolutely summable over $k$ in the trace class norm on $L^2(\R^d)$,
$$
\la D\ra\la\phi,W_{k,V}^{1,1}\ra\la D\ra^{-1},\qquad \la D\ra\la\phi,W_{k,V}^{1,2}\ra,\qquad \la\phi,W_{k,V}^{2,1}\ra\la D\ra^{-1},
$$
with the same bound \eqref{eqn:akbound} on the $\mathcal{L}^1$ norm of $W^{i,j}_{k,V}$ if $k\ge d+1$.
It follows that the sum of $\la\phi,W_{k,V}\ra$ converges absolutely in the trace class on $H^1\times L^2$.
\end{proof}


The trace of $\la\phi,W_{k,V}\ra$ is thus well defined, and the sum in \eqref{eqn:tracesum} converges in the trace class to $\la\phi,U_V-U_0\ra$. We next note that
$$
\trace_{H^1\times L^2}\la\phi,W_{k,V}\ra=
\trace_{H^1}\la\phi,W_{k,V}^{1,1}\ra+
\trace_{L^2}\la\phi,W_{k,V}^{2,2}\ra
=
2 \trace_{L^2}\la\phi,W_{k,V}^{1,1}\ra.
$$
The second inequality follows by noting that $W^{1,1}_{k,V}$ is the $L^2$ adjoint of $W^{2,2}_{k,V}$, and the trace of $W^{1,1}_{k,V}$ over $H^1(\R^d)$ is the same as the trace over $L^2(\R^d)$.

Since $W_{k,V}^{1,1}$ is even in $t$, we can assume that $\phi(t)$ is an even function of $t$  when deriving a formula for $\trace\la \phi,W_{k,V}\ra$.
We will establish the other parts of Theorem \ref{thm:main} as a consequence of the following three results.

\begin{theorem}\label{thm:traceW1}
Assume that $V\in L^\infty_{\comp}(\R^d)$ and $\phi\in C_{\comp}^\infty(\R)$ is even. Then
$$
\trace\la\phi, W_{1,V}\ra=\nu_d(\phi)\int V(x)\,dx,
$$
where, with
$c_d= 2^{3-d}\pi^{1-\frac d2}/\Gamma(\frac d2-1)$, and $D=-i\partial_t$,
$$
\nu_d(\phi)=
\begin{cases}
\displaystyle 2\int_0^\infty t\ts\phi(t)\,dt,& d=1,\\
\displaystyle 2\pi^{-1}\int_0^\infty \phi(t)\,dt,& d=2,\rule{0pt}{20pt}\\
c_d\bigl(D^{d-3}\phi\bigr)(0),& d\ge 3\;\;\text{odd},\rule{0pt}{17pt}\\
c_d\bigl(|D|^{d-3}\phi\bigr)(0),& d\ge 4\;\;\text{even}.\rule{0pt}{17pt}
\end{cases}
$$
For $d>1$ one also has the formulas,
$$
\nu_d(\phi)=
\begin{cases}
\displaystyle 2(-2\pi)^{\frac{1-d}2} \int_0^\infty 
\partial_t(t^{-1}\partial_t)^{\frac{d-3}2}\phi(t)\,dt,& d\ge 3\;\;\text{odd},\\
\displaystyle -4(-2\pi)^{-\frac d2} \int_0^\infty 
(t^{-1}\partial_t)^{\frac{d-2}2}\phi(t)\,dt,& d\ge 2\;\;\text{even}.\rule{0pt}{20pt}
\end{cases}
$$
\end{theorem}


\begin{theorem}\label{thm:traceW2}
Assume that $V\in L^\infty_{\comp}(\R^d)$. If $\phi\in C_{\comp}^\infty(\R)$ is even, then one can write
\begin{equation*}
\trace\la\phi,W_{2,V}\ra=
\begin{cases}
\displaystyle\int_0^\infty t^{4-d}a_{2,V}(t)\,\phi(t)\,dt,& 1\le d\le 4,\\
\displaystyle\int_0^\infty a_{2,V}(t)\,\partial_t(t^{-1}\partial_t)^{\frac{d-5}2}\phi(t)\,dt,& d\ge 5\;\text{odd},\rule{0pt}{20pt}\\
\displaystyle\int_0^\infty a_{2,V}(t)\,(t^{-1}\partial_t)^{\frac{d-4}2}\phi(t)\,dt,& d\ge 4\;\text{even},\rule{0pt}{20pt}
\end{cases}
\end{equation*}
where $a_{2,V}(t)\in C\bigl(\Rbarplus)$, and $|a_{2,V}(t)|\le (2\pi)^{-\lfloor\frac d2\rfloor}\|V\|_{L^2}^2$.
If $m$ is an integer and $V\in L^\infty_{\comp}\cap H^m(\R^d)$, then
$a_{2,V}(t)\in C^{2m}(\Rbarplus)$, and 
$$
\|a_{2,V}\|_{C^{2m}(\R)}\le C_{m,d}\, \|V\|_{H^m}^2.
$$
Additionally, for $0\le j\le 2m$,
\begin{equation}\label{eqn:traceW2}
\partial_t^ja_{2,V}(0)=\begin{cases} 0,& \;j\,\text{odd},\\
c_{j,d}\,\||D|^j V\|_{L^2}^2,& \;j\,\text{even},
\end{cases}
\end{equation}
where $c_{j,d}\ne 0$.

Conversely, if $V\in L^\infty_{\comp}(\R^d)$ is real valued, and there are constants $b_j\in \C$ and $C<\infty$ so that
\begin{equation}\label{eqn:traceW2'}
\biggl|a_{2,V}(t)-\sum_{j=0}^{m-1} b_j \ts t^{2j}\biggr|\le C\,t^{2m}\quad\text{for}\quad 0\le t\le 1,
\end{equation}
then $V\in H^{m}(\R^d)$, and hence \eqref{eqn:traceW2} holds.
\end{theorem}


\begin{theorem}\label{thm:traceWk}
Assume that $V\in X_d$, and $\phi\in C_{\comp}^\infty(\R)$ is an even function. Then if $k\ge 3$ and $2k\ge d$, there exists $a_{k,V}\in C\bigl(\Rbarplus)$ such that
$$
\trace\la\phi,W_{k,V}\ra=\int_0^\infty t^{2k-d}a_{k,V}(t)\,\phi(t)\,dt.
$$
If $k\ge 3$ and $2k<d$,
there exists $a_{k,V}\in C\bigl(\Rbarplus)$ such that
\begin{equation}\label{eqn:traceWk'}
\trace\la\phi,W_{k,V}\ra=
\begin{cases}
\displaystyle\int_0^\infty a_{k,V}(t)\,\partial_t\bigl(t^{-1}\partial_t\bigr)^{\frac{d-2k-1}2}\phi(t)\,dt,& d\ge 7\;\text{odd},\\
\displaystyle\int_0^\infty a_{k,V}(t)\,\bigl(t^{-1}\partial_t\bigr)^{\frac{d-2k}2}\phi(t)\,dt,& d\ge 6\;\text{even}\rule{0pt}{20pt}.
\end{cases}
\end{equation}
In each of the above cases, if $m\ge 0$ is an integer and $V\in X_d\cap H^m(\R^d)$, then
$a_{k,V}\in C^{2m}(\Rbarplus)$, and $\partial_t^j a_{k,V}(0)=0$ if $j<2m$ is odd.
\end{theorem}

We show here how Theorem \ref{thm:main} follows from the above three theorems, together with bounds \eqref{eqn:alphabound} and \eqref{eqn:alphabound3}. It follows from Lemma \ref{lem:tracesum} that the operator $\la\phi,U_V-U_0\ra$ is trace class, and the proof of Lemma \ref{lem:tracesum} shows that its trace determines a distribution in $\phi$, and that
$$
\trace\la\phi,U_V-U_0\ra=\sum_{k=1}^\infty (-1)^k\trace\la\phi,W_{k,V}\ra.
$$

If $d\le 4$, then the expression for $\mu_{d,V}(\phi)$ holds with  
$$\alpha_V(t)=\sum_{k=2}^\infty (-1)^ka_{k,V}(t)t^{2k-4}.
$$
The bounds \eqref{eqn:alphabound} and \eqref{eqn:alphabound3} show that the sum converges absolutely in $C^{2m}$ when $V\in X_d\cap H^m$, and that
\begin{equation}\label{eqn:alphajbound}
|\partial_t^j\alpha_V(t)|\le C_d\,\|V\|_{H^{\lceil j/2\rceil}}^2p_j\bigl(t\|V\|_{X_d}^{1/2}\bigr)
\cosh\bigl(t\|V\|_{X_d}^{1/2}\bigr),
\end{equation}
where $p_j$ is a polynomial of order at most $\max(0,j-2)$.

For $d\ge 5$ and $k\ge 3$, we need the following lemma.

\begin{lemma}\label{lem:tildeak}
If $F\in C(\overline{\R_+})$, and $\phi(t)$ is an even Schwartz function, then for each $m\ge 1$ and $n\ge 0$ one can write
\begin{align*}
\int_0^\infty t^n\,F(t)\,\phi(t)\,dt&=\int_0^\infty t^{n+2m-1}F_{m,n}(t)\partial_t(t^{-1}\partial_t)^{m-1}\phi(t)\,dt\\
&=\int_0^\infty t^{n+2m}F_{m,n}(t)(t^{-1}\partial_t)^m\phi(t)\,dt,
\end{align*}
where $F_{m,n}\in C(\overline{\R_+})$. If $F\in C^j(\overline{\R_+})$ then so if $F_{m,n}$, and
$$
\|\partial_t^j F_{m,n}\|_{L^\infty}\le 
\frac{\Gamma(\frac{n+j+1}2)}{2^m\Gamma(m+\frac{n+j+1}2)}\,\|\partial_t^j F\|_{L^\infty}.
$$
\end{lemma}
\begin{proof}
The $F_{m,n}$ are obtained recursively by the formula
$$
F_{m+1,n}(t)=-\int_0^1 s^{n+2m}F_{m,n}(st)\,ds,
$$
where $F_{0,n}=F$, and the lemma follows by induction.
\end{proof}
We observe that
\begin{equation}\label{eqn:Fmnderiv}
\partial_t^j F_{m,n}(0)=
\frac{(-1)^m\Gamma(\frac{n+j+1}2)}{2^m\Gamma(m+\frac{n+j+1}2)}\partial_t^j F(0).
\end{equation}

For $d\ge 6$ even, define $\tilde a_{k,V}$ as follows, where $F=a_{k,V}$ in each instance,
\begin{equation}\label{eqn:tildeakeven}
\tilde a_{k,V}(t)=
\begin{cases}
F_{k-2,0}(t),& 3\le k <\frac d2,\\
F_{\frac{d-4}2,2k-d}(t),& k\ge\frac d2.
\end{cases}
\end{equation}
For $d\ge 5$ odd, with $F=a_{k,V}$ we define
\begin{equation}\label{eqn:tildeakodd}
\tilde a_{k,V}(t)=
\begin{cases}
F_{k-2,1}(t),& 3\le k <\frac d2,\\
F_{\frac{d-3}2,2k-d}(t),& k>\frac d2.
\end{cases}
\end{equation}
Theorem \ref{thm:main} then holds with
$$
\alpha_V(t)=a_{2,V}(t)+\sum_{k=3}^\infty (-1)^k t^{2(k-2)}\tilde a_{k,V}(t).
$$
The bounds \eqref{eqn:alphajbound} on $\partial_t^j \alpha_V(t)$ again follow from \eqref{eqn:alphabound}.

The fact that $\partial_t^{2j+1}\alpha_V(0)=0$ for $0\le j<m$ follows from \eqref{eqn:traceW2}, Theorem \ref{thm:traceWk}, and \eqref{eqn:Fmnderiv}.
The particular form of the even order Taylor coefficients of $\alpha_V(t)$ at $t=0$ is discussed in Remark \ref{rem:heat} below.

For the final conclusion of Theorem \ref{thm:main}, assume $V\in X_d$, and that the expansion \eqref{eqn:traceW'} holds.
By induction we may assume $V\in H^{m-1}(\R^d)$. It follows by the above steps that
$$
\sum_{k=3}^\infty (-1)^k t^{2(k-3)}\tilde a_{k,V}(t)\in C^{2(m-1)}(\overline{\R_+}),
$$
and hence that we can write
$$
\alpha_V(t)-a_{2,V}(t)=\sum_{j=0}^{m-1} \tilde b_j t^{2j}+\mathcal{O}(t^{2m}).
$$ 
We conclude that $a_{2,V}(t)$ admits an expansion of the form \eqref{eqn:traceW2'}, and from Theorem \ref{thm:traceW2} that $V\in H^m(\R^d)$.

\begin{remark}\label{rem:heat}
As observed in \cite{SaBZw}, one can use the Fourier transform of the Gaussian to express the relative trace of the heat kernel for $-\Delta+V$ in terms of the relative trace of the wave group. For $t>0$, the following holds
\begin{equation*}
\trace\Bigl(e^{t(\Delta-V)}-e^{t\Delta}\Bigr)=
\frac 1{4\sqrt{\pi t}}\trace\int_\R e^{-\frac{s^2}{4t}}\bigl(U_V(s)-U_0(s)\bigr)\,ds.
\end{equation*}
Although the Gaussian is not compactly supported, the estimates \eqref{eqn:aksumbound} and \eqref{eqn:trphiWk}, combined with a partition of unity to decompose the integral over $s$, show that the right side converges in the trace-class norm if $t>0$.

As noted in \cite{SaBZw}, if $V\in C_{\comp}^\infty(\R^d)$ this relates the coefficients $a_j(V)$ in \eqref{eqn:traceW} to the heat coefficients $c_j(V)$ for $-\Delta+V$. The relation also holds for finite order expansions under the assumptions in Theorem \ref{thm:main}, using the fact that for $t>0$,
$$
\frac 1{\sqrt{4\pi t}}\int_{-1}^1 e^{-\frac {s^2}{4t}}\,s^{2m}\,\epsilon(s)\,ds=t^m\tilde\epsilon(t),
$$
where $\lim_{t\rightarrow 0^+}\tilde\epsilon(t)=0$ if $\lim_{s\rightarrow 0}\epsilon(s)=0$.
We remark that our proof expresses the  $a_j(V)$ as multilinear integrals in $D^\alpha V$ with $|\alpha|\le j-2$, just as for the $c_j(V)$ in \cite{SmZw}.

In particular, Theorem \ref{thm:main} can be used to prove existence of finite order expansions of the relative trace for the heat kernel, but not vice versa. As mentioned above, the analogue of Theorem \ref{thm:main} for the heat kernel was established in \cite{SmZw}, using only the a priori assumption $V\in L^\infty_{\comp}$ in all dimensions, rather than $V\in \hatL_{\comp}$ if $d\ge 4$.
\end{remark}


\section{A formula for the trace of $W_{k,V}$.}\label{sec:traceform}

In this section we derive an integral formula for the trace of $\la\phi,W_{k,V}\ra$, which by the proof of Lemma \ref {lem:tracesum} equals the trace of
$$
4\int_{\R_+^{k+1}}\phi(s_1+\cdots+s_{k+1})\Cos(s_{k+1})V\Sin(s_k)V\cdots V \Sin(s_1)\,d^{k+1}s.
$$
Let $\rho$ be a radial, compactly supported function on $\R^d$ such that $\rho(y)\ge 0$ and $\int\rho(y)\,dy=1$.
By the proof of Lemma \ref {lem:tracesum}, we can write this as the limit as $\epsilon\rightarrow 0$ of the trace of the following operator,
$$
4\int_{\R_+^{k+1}}\phi(s_1+\cdots +s_{k+1})\Cos(s_{k+1})V\Sin_\epsilon(s_k)V\cdots V \Sin_\epsilon(s_1)\,d^{k+1}s,
$$
where $S_\epsilon(s_j)=\hat{\rho}(\epsilon D)S(s_j)$. For $\epsilon>0$ the integrand is trace class for each value of $s$.
Since the trace is cyclic, and the integral is invariant under $s_1\leftrightarrow s_{k+1}$, we can equate this to the trace of
$$
4\int_{\R_+^{k+1}}\phi(s_1+\cdots +s_{k+1})\Sin_\epsilon(s_k)V\cdots V \Sin_\epsilon(s_1+s_{k+1})V\,d^{k+1}s,
$$
which by a change of variables equals the trace of
$$
4\int_{\R_+^k}s_1\phi(s_1+\cdots +s_k)\Sin_\epsilon(s_k)V\cdots V \Sin_\epsilon(s_1)V\,d^ks.
$$
Introducing $\psi(s)=s\phi(s)$, the cyclic nature of the trace equates this with the trace of the following,
$$
\frac 4k\int_{\R_+^k}\psi(s_1+\cdots +s_k)\Sin_\epsilon(s_k)V\cdots V \Sin_\epsilon(s_1)V\,d^ks.
$$
By the Plancherel formula, for $\epsilon>0$ we can write its trace as
$$
\frac 4{k(2\pi)^{dk}}\int_{\R^{dk}}\int_{\R_+^k}\psi(s_1+\cdots +s_k)\prod_{j=1}^k
\hat\rho(\epsilon\xi_j)\frac{\sin(s_j|\xi_j|)}{|\xi_j|}\,\widehat{V}(\xi_j-\xi_{j-1})\,d^ks\,d\xi.
$$
We show here that this converges, as $\epsilon\rightarrow 0$, to its value at $\epsilon=0$, assuming $V\in L^\infty_{\comp}$ if $k=1,2$, and $V\in \hatL_{\comp}$ if $k\ge 3$. The case of $V\in L^\infty_{\comp}$ and $1\le d\le 3$ will be handled in a later section.

By the proof of Lemma \ref{lem:tracesum}, the following function
$$
\int_{\R_+^k}\psi(s_1+\cdots +s_k)\prod_{j=1}^k
\frac{\sin(s_j|\xi_j|)}{|\xi_j|}\,d^ks
$$
can be written as a sum of functions, each of which is dominated for some value of $j$
by $(1+|\xi_j|)^{-d+1}$. Hence, it suffices by dominated convergence to show that
$$
\sup_{\xi_1}\int \prod_{j=1}^k |\widehat{V}(\xi_j-\xi_{j-1})|\,d\xi_2\cdots d\xi_k\le C.
$$
For $k=1$ this holds since $V\in L^1$, and for $k=2$ it holds since $V\in L^2$. For $k\ge 3$ it holds assuming $\widehat V\in L^1\cap L^2$ by Young's inequality. We thus conclude the following.

\begin{lemma}\label{lem:ktraceform}
With $\psi(s)=s\phi(s)$, the trace of $\la \phi,W_{k,V}\ra$ is equal to
\begin{equation}\label{eqn:ktraceform}
\frac 4{k(2\pi)^{dk}}\int_{\R^{dk}}\int_{\R_+^k}\psi(s_1+\cdots +s_k)\prod_{j=1}^k\frac{\sin(s_j|\xi_j|)}{|\xi_j|}\,\widehat{V}(\xi_j-\xi_{j-1})\,d^ks\,d\xi,
\end{equation}
where we assume that $V\in \hatL_{\comp}(\R^d)$ if $k\ge 3$, and $V\in L^2_{\comp}(\R^d)$ if $k=1, 2$.
\end{lemma}


\section{Proof of Theorem \ref{thm:traceW1}}
For $k=1$, we use \eqref{eqn:ktraceform} to write
$$
\trace\la \phi,W_{1,V}\ra=
\frac 4{(2\pi)^d}\int_{\R^d}
\int_0^\infty \psi(s)|\xi|^{-1}\sin(s|\xi|)\widehat{V}(0)\,ds\,d\xi.
$$
With $\psi(s)=s\phi(s)$ considered as a smooth, odd function on $\R$, this equals
\begin{multline*}
i\ts\Bigl(\int V(x)\,dx\Bigr)\,
\frac{2}{(2\pi)^d}\int_{\R^d} \hat\psi(|\xi|)|\xi|^{-1}\,d\xi\\
=
\Bigl(\int V(x)\,dx\Bigr)\,
\frac{2\omega_d}{(2\pi)^d}\int_0^\infty i\tau^{d-2}\hat\psi(\tau)\,d\tau,
\end{multline*}
where $\omega_d=2\pi^{\frac d2}/\Gamma(\frac d2)$ is the area of the $d-1$ sphere.

If $d=1$, then
$$
\int_0^\infty i\tau^{-1}\hat\psi(\tau)\,d\tau=\pi\int_0^\infty \psi(t)\,dt=
\pi\int_0^\infty t\ts\phi(t)\,dt,
$$
and if $d\ge 3$ is odd, then
$$
\int_0^\infty i\tau^{d-2}\hat\psi(\tau)\,d\tau=
i\pi \bigl(D^{d-2}\psi\bigr)(0)=
(d-2)\pi\bigl(D^{d-3}\phi\bigr)(0).
$$
The Taylor expansion of $\phi$ about $t=0$ shows that, for $\phi$ even in $t$,
$$
(d-2)\pi D^{d-3}\phi(0)=(-2)^{\frac{d-1}2}\pi^{\frac 12}\Gamma(\tfrac d2)
\int_0^\infty \partial_t(t^{-1}\partial_t)^{\frac{d-3}2}\phi(t)\,dt.
$$
For $d\ge 3$ odd this yields, assuming $\phi(t)$ is even in $t$,
$$
\nu_d(\phi)=2(-2\pi)^{\frac{1-d}2} \int_0^\infty \partial_t(t^{-1}\partial_t)^{\frac{d-3}2}\phi(t)\,dt.
$$

If $d=2$, then
$$
\int_0^\infty i\hat\psi(\tau)\,d\tau=
-\int_0^\infty \partial_\tau\hat\phi(\tau)\,d\tau=
2\int_0^\infty\phi(t)\,dt,
$$
and if $d\ge 4$ is even, then
$$
\int_0^\infty i\tau^{d-2}\hat\psi(\tau)\,d\tau=
(d-2)\int_0^\infty \tau^{d-3}\hat\phi(\tau)\,d\tau=
(d-2)\pi\bigl(|D|^{d-3}\phi\bigr)(0).
$$
Note that, for $\phi(t)$ even in $t$,
$$
\int_{-\infty}^\infty t^{-1}\partial_t\phi(t)\,dt=-\pi (|D|\phi)(0).
$$
The following identity, which is valid for $k\ge 0$,
$$
t\partial_t|D|^k\delta=-(k+1)|D|^k\delta,
$$
shows that when $\phi(t)$ is even in $t$, and $k\ge 1$,
$$
\int_0^\infty (t^{-1}\partial_t)^k\phi(t)\,dt=\frac{(-1)^k\pi}{2^k(k-1)!}\bigl(|D|^{2k-1}\phi\bigr)(0).
$$
Hence when $d\ge 4$ is even, and $\phi(t)$ is even in $t$,
$$
\nu_d(\phi)=4(-1)^{\frac{d-2}2}(2\pi)^{-\frac d2} \int_0^\infty (t^{-1}\partial_t)^{\frac{d-2}2}\phi(t)\,dt.
$$
From the above formula, this holds when $d=2$ as well.


\section{Evaluation of the trace for $k\ge 2$ and $2k\ge d$.}
Let
$$
M_k(\xi)=\int_{\Delta^{k-1}}\;\prod_{j=1}^k\biggl(\,\frac{\sin(s_j|\xi_j|)}{|\xi_j|}\biggr)\,d\bfs,
$$
where $\Delta^{k-1}$ is the $k-1$-simplex,
$$
\Delta^{k-1}=\{\bfs\in\R^{k}:s_1+\cdots+s_k=1\,,\;s_j\ge 0 \;\; \forall j\},
$$
and the measure $d\bfs$ is the pullback of Lebesgue measure by the projection onto $(s_1,\ldots,s_{k-1})$.
Then by Lemma \ref{lem:ktraceform}, we can express $\trace\la \phi,W_{k,V}\ra$ as
\begin{equation*}
\frac{4}{k(2\pi)^{dk}}\int_{\R^{dk}}\int_0^\infty t^{2k}\ts\phi(t) M_k(t\xi)
\biggl(\prod_{j=1}^k\widehat{V}(\xi_j-\xi_{j-1})\biggr)\,dt\,d\xi.
\end{equation*}

\begin{lemma}\label{lem:Vhatprod}
Suppose $V_j\in \hatL\cap H^m(\R^d)$, and let $V_j^{(m_j)}=\la D\ra^{m_j} V_j$.
Assume that
$\{m_j\}\in\N^k$ satisfy $m_j\le m$, and $\sum_j m_j\le 2m$. Then
\begin{multline*}
\sup_{\xi_1}\int \prod_{j=1}^k\,\Bigl|\widehat{V_j^{(m_j)}}(\xi_j-\xi_{j-1})\Bigr|\,d\xi_2\cdots d\xi_k\\
\le (2\pi)^{d(k-1)}
\Bigl(\sup_j \|V_j\|_{\hatL}\Bigr)^{k-2}\Bigl(\sup_j \|V_j\|_{H^m}\Bigr)^2.
\end{multline*}
If $k=2$, the bound holds for $V_j\in H^m(\R^d)$.
\end{lemma}
\begin{proof}
Note that $|\widehat{V_j^{(m_j)}}(\xi)|=\la\xi\ra^{m_j}|\widehat{V_j}(\xi)|$.
By convexity, it suffices to consider the extremal cases where $m_j=m$ for 2 values of $j$ and $m_j=0$ otherwise,
thus $\widehat{V_j^{(m_j)}}\in L^2$ for two values of $j$, and $\widehat{V_j^{(m_j)}}\in L^1$ for all other values of $j$. The result then follows by Young's inequality.
\end{proof}

\begin{lemma}\label{lem:multform}
$$
M_k(\xi)=\,
\sum_{j=1}^k\biggl(\,\prod_{i\ne j}\frac{1}{|\xi_i|^2-|\xi_j|^2}\biggr)\frac{\sin(|\xi_j|)}{|\xi_j|}\,.
$$
\end{lemma}
\begin{proof}
We start with the equality
$$
\int_0^s\sin{((s-r)A)}\,\sin{(rB)}\,dr=\frac{B\sin(sA)}{B^2-A^2}+\frac{A\sin(sB)}{A^2-B^2}\,.
$$
Let $r_j=s_1+\cdots +s_j$ for $1\le j\le k-1$, and $r_0=0$, $r_k=1$. Repeated application of the above equality then yields
\begin{multline*}
\int_0^1\int_0^{r_{k-1}}\cdots\int_0^{r_2}\biggl(\,\prod_{j=1}^k\sin\bigl((r_j-r_{j-1})|\xi_j|\bigr)\biggr)\,dr_1\cdots dr_{k-1}\\
=\sum_{j=1}^k\biggl(\,\prod_{i\ne j}\frac{|\xi_i|}{|\xi_i|^2-|\xi_j|^2}\biggr)\sin(|\xi_j|),
\end{multline*}
and we obtain the result.
In evaluating the integral we use that
$$
\sum_{j=1}^{k-1}\frac{1}{|\xi_j|^2-|\xi_k|^2}\biggl(\,\prod_{i\ne j,i<k}\frac{1}{|\xi_i|^2-|\xi_j|^2}\biggr)=\prod_{i=1}^{k-1}\frac{1}{|\xi_i|^2-|\xi_k|^2}.
$$
This is equivalent to 
$$
\sum_{j=1}^{k-1}\biggl(\,\prod_{i\ne j,i<k}\frac{|\xi_i|^2-|\xi_k|^2}{|\xi_i|^2-|\xi_j|^2}\biggr)=1,
$$
which can be seen by observing that replacing $|\xi_k|^2$ by $z$ gives a polynomial of order $k-2$ in $z$ that has value 1 at each of $z=|\xi_n|^2$, $n=1,\ldots,k-1$.
\end{proof}

We introduce a family of functions $G_\nu(w)$ for $\nu>-1$, defined initially for $w\notin(-\infty,0]$ by the rule
$$
G_\nu(z^2)=\frac{\sqrt{\pi}}{(2z)^\nu}\,J_\nu(z),\quad \Re(z)>0,
$$
From the power series expansion of $J_\nu(z)$, see \cite[3.1(8)]{W}, we have
\begin{equation}\label{eqn:Gseries}
G_\nu(w)=4^{-\nu}\sqrt{\pi}\sum_{m=0}^\infty \frac{(-1)^m(\frac 14 w)^m}{m!\Gamma(\nu+m+1)}.
\end{equation}
We remark that the function $G_\nu$ is a generalized hypergeometric function,
$$
G_\nu(w)=\frac{4^{-\nu}\sqrt{\pi}}{\Gamma(\nu+1)}\,{_0F_1}\Bigl(\nu+1,-\frac{w}4\Bigr),
$$
though we will not use this fact here. See \cite[p. 362, 9.1.69]{AS}.

From a standard integral formula for $J_\nu(z)$, see \cite[6.2(8)]{W}, we see that $G_\nu(z)$ is an entire function given by the following integral, where $c>0$,
$$
G_\nu(z)=\frac{4^{-\nu}\sqrt{\pi}}{2\pi i}\int_{c-\infty i}^{c+\infty i}e^{s-\frac z{4s}}\frac{ds}{s^{\nu+1}}\;.
$$
From this we easily derive the following relations,
\begin{equation}\label{eqn:Gdiff}
\frac{d}{dz} G_\nu(z)=-G_{\nu+1}(z),
\end{equation}
and
\begin{equation}\label{eqn:Gint}
\int_0^\infty G_{\nu+\frac 12}(r^2+z)\,dr=\frac{\sqrt{\pi}}{2}\, G_{\nu}(z).
\end{equation}
We note the following special cases:
\begin{equation}\label{eqn:Gcases}
G_\nu(z^2)=\begin{cases} 
2\cos(z),& \nu=-\frac 12,\\
\sqrt{\pi} \ts J_0(z),& \nu=0,\\
z^{-1}\sin z,& \nu=\frac 12.
\end{cases}
\end{equation}

Another formula for $G_\nu(z^2)$, for $\nu>-1/2$, follows from \cite[3.3(4)]{W},
$$
G_\nu(z^2)=\frac{1}{4^\nu \ts \Gamma(\nu+\frac 12)}\int_{-1}^1 (1-s^2)^{\nu-\hf}\,e^{izs}\,ds.
$$
This leads to the following bound for $j\ge 0$ and $x\in\R$, which is also directly verified for $\nu=-1/2$ using \eqref{eqn:Gcases}.
\begin{equation}\label{eqn:Gbound}
\begin{split}
\sup_{x\ge 0}\;\Bigl|\frac{d^j}{dx^j}G_\nu(x^2)\Bigr|&= \frac{2}{4^\nu \ts \Gamma(\nu+\frac 12)}\int_0^1 s^j(1-s^2)^{\nu-\hf}\,ds\\
&=\frac{\Gamma(\frac{j+1}2)}{4^\nu\ts\Gamma(\frac j2+\nu+1)}.
\end{split}
\end{equation}

\begin{lemma}\label{lem:multform2}
For any smooth function $f(x_1,\ldots,x_k):\R^k\rightarrow\C$, one has
$$
\sum_{j=1}^k f(x_j)\biggl(\,\prod_{i\ne j}\frac{1}{x_i-x_j}\biggr)
=
(-1)^{k-1}\int_{\Delta^{k-1}}f^{(k-1)}(s_1x_1+\cdots+s_k x_k)\,d\bfs.
$$
\end{lemma}
\begin{proof}
For $k=2$, this reduces to the equality
$$
\int_0^1 f'(s_1x_1+(1-s_1)x_2)\,ds_1=\frac{f(x_1)-f(x_2)}{x_1-x_2}.
$$
We proceed by induction, and write
\begin{align*}
&\int_{\Delta^k}f^{(k)}(s_1x_1+\cdots+s_kx_k+s_{k+1}x_{k+1})\,d\bfs\\
&=\frac{1}{x_k-x_{k+1}}\int_{\Delta^{k-1}}\Bigl(f^{(k-1)}(s_1x_1+\cdots s_k x_k)-
f^{(k-1)}(s_1x_1+\cdots s_k x_{k+1})\Bigr)\,d\bfs
\end{align*}
By the induction assumption, the latter equals
\begin{multline*}
\frac{(-1)^{k-1}}{x_k-x_{k+1}}\sum_{j=1}^{k-1}f(x_j)\biggl(\,\prod_{\substack{i\ne j\\i<k}}\frac{1}{x_i-x_j}\biggr)\biggl(\frac{1}{x_k-x_j}-\frac{1}{x_{k+1}-x_j}\biggr)\\
+\frac{(-1)^{k-1}}{x_k-x_{k+1}}\biggl(f(x_k)\prod_{i<k}\frac{1}{x_i-x_k}-
f(x_{k+1})\prod_{i<k}\frac{1}{x_i-x_{k+1}}\biggr)\\
=(-1)^k\sum_{j=1}^{k+1}f(x_j)\biggl(\,\prod_{i\ne j}\frac{1}{x_i-x_j}\biggr).
\end{multline*}
\end{proof}

For $\bfs\in\Delta^{k-1}$ and $\xi\in \R^{dk}$, we will use the following notation,
$$
\bfs\cdot\xi=\sum_{j=1}^k s_j\xi_j,\qquad |\xi|_{\bfs}^2=\sum_{j=1}^k s_j |\xi_j|^2.
$$
We can write
\begin{align}
|\xi|_{\bfs}^2
&=|\bfs\cdot\xi|^2
+|\xi|_{\bfs}^2 - |\bfs\cdot\xi|^2\notag
\\
&\equiv|\bfs\cdot\xi|^2+Q_{k,\bfs}(\xi).\label{eqn:etaform}
\end{align}
Note that by strict convexity of $|\cdot|^2$, for $\bfs$ in the interior of $\Delta^{k-1}$ we have $Q_{k,\bfs}(\xi)>0$ unless $\xi_i=\xi_j$ for all $i,j$.
Furthermore,
\begin{equation}\label{eqn:Qtrans}
Q_{k,\bfs}(\xi_1,\xi_2,\ldots,\xi_k)=Q_{k,\bfs}(0,\xi_2-\xi_1,\ldots,\xi_k-\xi_1),
\end{equation}
and
\begin{equation}\label{eqn:Q2sform}
Q_{2,\bfs}(\xi)=s_1(1-s_1)|\xi_2-\xi_1|^2.
\end{equation}

Recalling \eqref{eqn:Gcases}, by Lemmas \ref{lem:multform} and \ref{lem:multform2} we have
\begin{equation*}
M_k(t\xi)=\int_{\Delta^{k-1}}G_{k-\hf}\bigl(t^2|\xi|_\bfs^2\bigr)\,d\bfs.
\end{equation*}
This leads to the following expression for $\trace\la \phi,W_{k,V}\ra$, for $\phi$ even,
\begin{equation*}
\frac{4}{k(2\pi)^{dk}}\int_{\Delta^{k-1}}\int_{\R^{dk}}\int_{0}^\infty t^{2k}\ts\phi(t)\,
G_{k-\hf}\bigl(t^2|\xi|_\bfs^2\bigr) 
\biggl(\prod_{j=1}^k\widehat{V}(\xi_j-\xi_{j-1})\biggr)\,dt\,d\xi\,d\bfs.
\end{equation*}

For $k\ge 1$ and $t\ge 0$, we can write $G_{k-\hf}\bigl(t^2|\xi|^2_{\bfs}\bigr)$ as
\begin{multline*}
\Bigl(\frac {-1}{2z}\frac{d}{dz}\Bigr)^{k-1}\Bigl(\frac{\sin z}{z}\Bigr)\Bigr|_{z=t|\xi|_\bfs}
\\
=\sum_{j=k}^{2k-1}\frac{1}{\bigl(t|\xi|_{\bfs}\bigr)^j}\Bigl(a_{k,j}\cos\bigl(t|\xi|_{\bfs}\bigr)+b_{k,j}\sin\bigl(t|\xi|_{\bfs}\bigr)\Bigr),
\end{multline*}
which shows that we can write 
$$
\int_{0}^\infty t^{2k}\ts\phi(t)\,
G_{k-\hf}\bigl(t^2|\xi|_\bfs^2\bigr)\,dt=\Phi(|\xi|_{\bfs}\bigr),
$$
for an even, Schwartz function $\Phi$.

We make the volume preserving change of variables:
\begin{equation*}
\eta_1=\bfs\cdot\xi,\quad \eta_j=\xi_j-\xi_1\;\;\text{for}\;\; 2\le j\le k.
\end{equation*}
Let $Q_{k,\bfs}(\eta')=Q_{k,\bfs}(0,\eta_2,\ldots,\eta_k)$. Then by \eqref{eqn:Qtrans}, the integrand becomes
\begin{multline*}
t^{2k}\phi(t)G_{k-\hf}\Bigl(t^2|\eta_1|^2+t^2Q_{k,\bfs}(\eta')\Bigr)
\widehat{V}(\eta_2)\widehat{V}(\eta_3-\eta_2)\times\\
\cdots\widehat{V}(\eta_k-\eta_{k-1})\widehat{V}(-\eta_k).
\end{multline*}
This is radial in $\eta_1$. Consequently, we can express $\trace\la \phi,W_{k,V}\ra$ as
\begin{multline}\label{eqn:traceform}
\frac{8\pi^{\frac d2}}{k(2\pi)^{kd}\Gamma(\frac d2)}\int_{\Delta^{k-1}}\int_{\R^{d(k-1)}}\int_0^\infty\int_0^\infty t^{2k}\ts\phi(t)\times\\
G_{k-\frac 12}\bigl(t^2r^2+t^2Q_{k,\bfs}(\eta')\bigr)\widehat{V}(\eta_2)
\cdots\widehat{V}(-\eta_k)\,dt\,r^{d-1}dr\,d\eta'\,d\bfs.
\end{multline}
By Lemma \ref{lem:Vhatprod}, we have that
$$
\int \bigl|\widehat{V}(\eta_2)\widehat{V}(\eta_3-\eta_2)
\cdots\widehat{V}(-\eta_k)\bigr|\,d\eta'\le
(2\pi)^{d(k-1)}\|V\|_{\hatL}^{k-2}\|V\|_{L^2}^2.
$$
Assume now that $2k\ge d$. The above integral is oscillatory, and requires integration in $t$ before $r$ to be convergent. We proceed formally here and flip the integration order of $t$ and $r$, but note that the following steps can be made rigorous by inserting a cutoff $\chi(\epsilon tr)$ and letting $\epsilon\rightarrow 0$.
Write
\begin{multline*}
t^d\int_0^\infty G_{k-\hf}\bigl(t^2r^2+t^2Q_{k,\bfs}(\eta')\bigr)\,r^{d-1}\,dr=\\
\int_0^\infty \Bigl(\frac {-1}{2r}\frac \partial{\partial r}\Bigr)^m G_{k-m-\hf}\bigl(r^2+t^2Q_{k,\bfs}(\eta')\bigr)\,r^{d-1}dr.
\end{multline*}
If $d$ is even, we take $m=\frac d2$ and integrate by parts to obtain
$$
\frac{\Gamma(\frac d2)}{2} G_{k-\frac{d+1}2}\bigl(t^2Q_{k,\bfs}(\eta')\bigr).
$$
If $d$ is odd, we take $m=\frac {d-1}2$, integrate by parts, and use \eqref{eqn:Gint} to obtain
$$
\frac{\Gamma(\frac d2)}{\Gamma(\frac 12)}\int_0^\infty G_{k-\frac d2}\bigl(r^2+t^2Q_{k,\bfs}(\eta')\bigr)\,dr
=\frac{\Gamma(\frac d2)}{2}G_{k-\frac{d+1}2}\bigl(t^2Q_{k,\bfs}(\eta')\bigr).
$$

We thus conclude the following.

\begin{lemma}\label{lem:akVform}
If $V\in \hatL_{\comp}(\R^d)$ and $2k\ge d$, then 
$$\trace\la \phi,W_{k,V}\ra=\int_0^\infty t^{2k-d}a_{k,V}(t)\ts \phi(t)\,dt$$
where
\begin{multline}\label{eqn:akVform}
a_{k,V}(t)=\frac{4\pi^{\frac d2}}{k(2\pi)^{dk}}
\int_{\R^{d(k-1)}}\int_{\Delta^{k-1}}
G_{k-\frac {d+1}2}\Bigl(t^2Q_{k,\bfs}(\eta')\Bigr)\times\\
\widehat{V}(\eta_2)\widehat{V}(\eta_3-\eta_2)\cdots\widehat{V}(-\eta_k)\,d\bfs\,d\eta'.
\end{multline}
\end{lemma}

We conclude this section with the following observation.
\begin{lemma}\label{lem:Gkzero}
If $2k\ge d$, then
$$
G_{k-\frac {d+1}2}(0)=
\begin{cases}
\frac{\pi^{\frac 12}}{2^{2k-d-1}(k-\frac{d+1}2)!},& \;d\;\text{odd},\\
\frac{2(k-\frac d2)!}{(2k-d)!},& \;d\;\text{even}.
\end{cases}
$$
\end{lemma}
\begin{proof}
If $d$ is even, we use \eqref{eqn:Gdiff} and \eqref{eqn:Gcases} to write
$$
G_{k-\frac {d+1}2}(0)=2\Bigl(\frac{-1}{2r}\frac d{dr}\Bigr)^{k-\frac d2}\cos(r)\Bigr|_{r=0}.
$$
If $d$ is odd, we use \eqref{eqn:Gdiff} and \eqref{eqn:Gcases} to write
$$
G_{k-\frac {d+1}2}(0)=\sqrt{\pi} \Bigl(\frac{-1}{2r}\frac d{dr}\Bigr)^{k-\frac {d+1}2}J_0(r)\Bigr|_{r=0}.
$$
The Taylor expansion for $J_0(r)$ is given by
$$
J_0(r)=\sum_{j=0}^\infty \frac{(-1)^j}{(j!)^2}\Bigl(\frac r2\Bigr)^{2j},
$$
which leads to the desired formula.
\end{proof}


\section{Evaluation of the trace for $k\ge 2$ and $2k < d$.}

$\bullet$ If $2k<d$ we cannot use Lemma \ref{lem:akVform}. Instead, we recall \eqref{eqn:traceform} and write
$$
\trace\la\phi,W_{k,V}\ra=\int M_{\phi}(\eta')\,\widehat{V}(\eta_2)\widehat{V}(\eta_3-\eta_2)\cdots\widehat{V}(\eta_k-\eta_{k-1})\widehat{V}(-\eta_k)\,d\eta',
$$
where
\begin{multline*}
M_{\phi}(\eta')=\frac{8\pi^{\frac d2}}{k(2\pi)^{kd}\Gamma(\frac d2)}\times\\
\int_{\Delta^{k-1}}\int_0^\infty\int_0^\infty t^{2k}\ts\phi(t)
G_{k-\frac 12}\bigl(t^2r^2+t^2Q_{k,\bfs}(\eta')\bigr)\,dt\,r^{d-1}dr\,d\bfs.
\end{multline*}

By \eqref{eqn:Gdiff} and \eqref{eqn:Gcases} we can write
$$
t^{2k} G_{k-\frac 12}\bigl(t^2 r^2+t^2Q_{k,\bfs}(\eta')\bigr)=2 \Bigl(\frac {-1}{2r} \frac d{dr}\Bigr)^k\cos\Bigl(t\sqrt{r^2+Q_{k,\bfs}(\eta')}\,\Bigr).
$$
Since $\phi\in\mathcal{S}(\R)$ we can integrate by parts in $r$ to see that $M_\phi(\eta')$ equals
\begin{align*}
&\frac{16\pi^{\frac d2}}{k(2\pi)^{kd}\Gamma(\frac {d-2k}2)}\int_{\Delta^{k-1}}\int_0^\infty\int_0^\infty \phi(t)
\cos\Bigl(t\sqrt{r^2+Q_{k,\bfs}(\eta')}\,\Bigr)\,dt\,r^{d-2k-1}dr\,d\bfs\\
=\,&\frac{8\pi^{\frac d2}}{k(2\pi)^{kd}\Gamma(\frac {d-2k}2)}\int_{\Delta^{k-1}}\int_0^\infty
\widehat\phi\Bigl(\sqrt{r^2+Q_{k,\bfs}(\eta')}\,\Bigr)\,r^{d-2k-1}dr\,d\bfs.
\end{align*}
A simple calculation shows the following, where $\phi\in\mathcal{S}(\R)$ is even,
$$
\frac 1\sigma\frac{d}{d\sigma}\widehat{\Bigl(\frac 1t\frac{d\phi}{dt}\Bigr)}(\sigma)=\widehat\phi(\sigma).
$$
Hence, when $r\ge 0$,
$$
\Bigl(\frac 1r\frac{d}{dr}\Bigr)^j\Bigl(\Bigl(\frac 1t\frac d{dt}\Bigr)^j\phi\Bigr)^{\wedge}\Bigl(\sqrt{r^2+c}\,\Bigr)
=\widehat\phi\Bigl(\sqrt{r^2+c}\,\Bigr).
$$
Consider first the case that $d$ is even. We apply this with $j=\frac{d-2k}2$, and integrate by parts to see that $M_\phi(\eta')$ equals the following,
\begin{align*}
&\frac{(-2)^{\frac{d-2k-2}2}8\pi^{\frac d2}}{k(2\pi)^{kd}}\int_{\Delta^{k-1}}\int_0^\infty
\frac{d}{dr}\bigl((t^{-1}\partial_t)^{\frac{d-2k}2}\phi\bigr)^{\wedge}\Bigl(\sqrt{r^2+Q_{k,\bfs}(\eta')}\,\Bigr)\,dr\,d\bfs\\
=\,&\frac{4(-1)^{\frac {d-2k}2}}{k2^k(2\pi)^{(k-\frac 12)d}}\int_{\Delta^{k-1}}
\bigl((t^{-1}\partial_t)^{\frac{d-2k}2}\phi\bigr)^{\wedge}\Bigl(\sqrt{Q_{k,\bfs}(\eta')}\,\Bigr)\,d\bfs.
\end{align*}
Thus, when $d\ge 2k$ is even, we can write
$$
\trace\la\phi,W_{k,V}\ra=\int_0^\infty a_{k,V}(t)(t^{-1}\partial_t)^{\frac{d-2k}2}\phi(t)\,dt,
$$
where 
\begin{multline}\label{eqn:tracekeven}
a_{k,V}(t)=\frac{8(-1)^{\frac {d-2k}2}}{k2^k(2\pi)^{(k-\frac 12)d}}\times\\
\int_{\R^{(k-1)d}}\int_{\Delta^{k-1}}
\cos\Bigl(t\sqrt{Q_{k,\bfs}(\eta')}\,\Bigr)
\widehat{V}(\eta_2)\widehat{V}(\eta_3-\eta_2)\cdots\widehat{V}(-\eta_k)
\,d\bfs\,d\eta'.
\end{multline}

Now suppose that $d$ is odd. Similar steps, taking $j=\frac{d-2k-1}2$, lead to the following formula for $M_\phi(\eta')$,
\begin{equation*}
\frac{8(-1)^{\frac{d-2k-1}2}}{k2^k(2\pi)^{(k-\frac 12)d+\frac 12}}\int_{\Delta^{k-1}}\int_0^\infty
\bigl((t^{-1}\partial_t)^{\frac{d-2k-1}2}\phi\bigr)^{\wedge}\Bigl(\sqrt{r^2+Q_{k,\bfs}(\eta')}\,\Bigr)\,dr\,d\bfs.
\end{equation*}
If $f\in\mathcal{S}(\R)$ is an even function, for $c \ge 0$ we have
\begin{align*}
\hat f\bigl(\sqrt{r^2+c}\,\bigr)&=-2\int_0^\infty f'(t)\,\frac{\sin\bigl(t\sqrt{r^2+c}\,\bigr)}{\sqrt{r^2+c}}\,dt\\
&=-2\int_0^\infty tf'(t) \, G_{\hf}\bigl(t^2r^2+t^2c\bigr)\,dt.
\end{align*}
Thus, by \eqref{eqn:Gint} and \eqref{eqn:Gcases}, we can write
$$
\int_0^\infty \hat f\bigl(\sqrt{r^2+c}\,\bigr)\,dr=-\pi\int_0^\infty f'(t)J_0\bigl(t\sqrt{c}\ts\bigr)\,dt.
$$
Consequently, when $d>2k$ is odd, we can write
$$
\trace\la\phi,W_{k,V}\ra=\int_0^\infty a_{k,V}(t)\,\partial_t(t^{-1}\partial_t)^{\frac{d-2k-1}2}\phi(t)\,dt,
$$
where 
\begin{multline}\label{eqn:tracekodd}
a_{k,V}(t)=\frac{4(-1)^{\frac {d-2k+1}2}}{k2^k(2\pi)^{(k-\frac 12)d-\frac 12}}\times\\
\int_{\R^{(k-1)d}}\int_{\Delta^{k-1}}
J_0\Bigl(t\sqrt{Q_{k,\bfs}(\eta')}\Bigr)
\widehat{V}(\eta_2)\widehat{V}(\eta_3-\eta_2)\cdots\widehat{V}(-\eta_k)
\,d\bfs\,d\eta'.
\end{multline}


\section{Proof of Theorem \ref{thm:traceW2}}

Recall that if $k=2$, then $Q_{2,\bfs}(\eta)=s_1(1-s_1)|\eta|^2$, where $\eta\in\R^d$.
We begin the proof of Theorem \ref{thm:traceW2} by analyzing the form of $a_{2,V}(t)$, and
consider first the case of $1\le d\le 3$, hence $2k \ge d$.

$\bullet$ If $d=3$: we obtain from \eqref{eqn:akVform}, \eqref{eqn:Gcases}, and \eqref{eqn:Q2sform} that
$$
a_{2,V}(t)=\frac{1}{2(2\pi)^4}
\int_{\R^3}\int_0^1
J_0\Bigl(t|\eta|\sqrt{s-s^2}\Bigr)
|\widehat{V}(\eta)|^2\,ds\,d\eta.
$$
We use the Taylor expansion of $J_0(z)$ to evaluate the integral,
\begin{align}\notag
\int_0^1 J_0\Bigl(t|\eta|\sqrt{s-s^2}\Bigr)ds&=
\sum_{j=0}^\infty \frac{(-1)^j}{(j!)^2}\biggl(\int_0^1 s^j(1-s)^j\,ds\biggr)\biggl(\frac{t|\eta|}2\biggr)^{2j}\\
&=\frac{\sin(\frac 12 t|\eta|)}{\frac 12 t|\eta|},\label{eqn:J0int}
\end{align}
where we use the following special case of the beta integral
$$
\int_0^1 s^j(1-s)^j\,ds=\frac{(j!)^2}{(2j+1)!}\,.
$$
Thus,
\begin{equation*}
a_{2,V}(t)=\frac{1}{(2\pi)^4}
\int_{\R^3}
\frac{\sin\bigl(\frac 12 t|\eta|\bigr)}{t|\eta|}\,
|\widehat{V}(\eta)|^2\,d\eta,\qquad d=3.
\end{equation*}

$\bullet$ If $d=1$, then by \eqref{eqn:akVform},
$$
a_{2,V}(t)=\frac{2\pi^{\frac 12}}{(2\pi)^2}
\int_{\R}\int_0^1
G_1\bigl(t^2s(1-s)|\eta|^2\bigr)
|\widehat{V}(\eta)|^2\,ds\,d\eta.
$$
Using \eqref{eqn:Gcases}, \eqref{eqn:Gdiff}, and the expansion of $J_0$, this yields
$$
a_{2,V}(t)=\frac{1}{\pi}
\int_{\R}
\frac{1-\cos\bigl(\frac 12 t|\eta|\bigr)}{t^2|\eta|^2}\,
|\widehat{V}(\eta)|^2\,d\eta,\qquad d=1.
$$

$\bullet$ If $d=2$, then by \eqref{eqn:akVform},
$$
a_{2,V}(t)=\frac{1}{(2\pi)^3}
\int_{\R^2}\int_0^1
G_{\hf}\bigl(s(1-s)t^2|\eta|^2\bigr)
|\widehat{V}(\eta)|^2\,ds\,d\eta.
$$

$\bullet$ If $d\ge 5$ is odd, we use \eqref{eqn:tracekodd} and \eqref{eqn:J0int} to write $\trace\la\phi,W_{2,V}\ra$ in the indicated form with
\begin{equation}\label{eqn:2traceform}
a_{2,V}(t)=
\frac{(-1)^{\frac{d+1}2}}{(2\pi)^{\frac {3d-1}2}}
\int_{\R^d}\frac{\sin\bigl(\frac 12 t|\eta|\bigr)}{t|\eta|}\,|\widehat{V}(\eta)|^2\,d\eta.
\end{equation}
Note that the same form for $a_{2,V}(t)$ holds when $d=3$, which is expected by formally writing
$\partial_t(t\ts \partial_t)^{-1}\phi=t\ts \phi$.

$\bullet$ If $d\ge 4$ is even, we use \eqref{eqn:tracekeven} to write $\trace\la\phi,W_{2,V}\ra$ in the indicated form with
\begin{equation}\label{eqn:trace2even}
a_{2,V}(t)=
\frac{(-1)^{\frac d2}}{(2\pi)^{\frac {3d}2}}
\int_{\R^d}\int_0^1 \cos\Bigl(t|\eta|\sqrt{s-s^2}\,\Bigr)|\widehat V(\eta)|^2\,ds\,d\eta.
\end{equation}


\begin{proof}[Proof of Theorem \ref{thm:traceW2}]
We present the proof for $d\ge 3$; the proof for $d=1,2$ follows similarly. We assume $V\in L^\infty_{\comp}(\R^d)$ is real valued.

If $d\ge 3$ is odd, by \eqref{eqn:2traceform} the asserted form for the trace holds with
\begin{equation*}
a_{2,V}(t)=\frac{(-1)^{\frac{d+1}2}}{2(2\pi)^{\frac {3d-1}2}}\int_{\R^d}\,
G_{\hf}\bigl(t^2|\eta|^2/4\bigr)\,|\widehat V(\eta)|^2\,d\eta.
\end{equation*}
Since $|\widehat V(\eta)|^2$ is integrable, dominated convergence implies $a_{2,V}\in C^0\bigl(\Rbarplus\bigr)$. Furthermore, $|a_{2,V}(t)|\le \|V\|_{L^2}^2/2(2\pi)^{\frac{d-1}2}$,
and equality holds at $t=0$.

Assume now that $V\in H^m(\R^d)$, where $m\ge 1$. Dominated convergence and differentiation under the integral shows that $a_{2,V}\in C^{2m}(\overline{\R_+})$. Furthermore, by \eqref{eqn:Gbound},
\begin{equation*}
|\partial_t^j a_{2,V}(t)|\le \frac{1}{2^{j+1}(2\pi)^{\frac{d-1}2}(j+1)}\bigl\||D|^{\frac j2}V\|_{L^2}^2.
\end{equation*}
We note that this bound also holds for $d=1$.
For $m\ge 1$ we can write
$$
G_{\hf}\bigl(t^2|\eta|^2/4\bigr)=
\sum_{j=0}^{m-1}\frac{(-1)^jt^{2j}|\eta|^{2j}}{2^{2j}(2j+1)!}
\,+\,
(-1)^m r_{2m+1}(t|\eta|)\,t^{2m}|\eta|^{2m},
$$
where $r_{2m+1}(s)$ is a nonnegative, even, real analytic function satisfying
$$
r_{2m+1}(0)=\frac{1}{2^{2j}(2m+1)!}\,,\qquad \sup_{s\in\R}\,r_{2m+1}(s)\le\frac{1}{2^{2j}(2m+1)!}\,.
$$
By dominated convergence it follows that the expansion \eqref{eqn:traceW2} holds, with
\begin{equation}\label{eqn:c2jform}
c_{2,j}=\frac{(-1)^{j+\frac{d+1}2}\||D|^jV\|_{L^2}^2}{ 2^{2j+1} (2\pi)^{\frac{d-1}2} (2j+1)!}\,.
\end{equation}
Conversely, suppose that $V\in H^{m-1}$, and that \eqref{eqn:traceW2'} holds. Necessarily \eqref{eqn:c2jform} holds for $0\le j\le m-1$ by the above steps. We thus conclude that
$$
\sup_{t\in[0,1]}\;\int_{\R^3}r_{2m+1}(t|\eta|)\,|\eta|^{2m}
|\widehat V(\eta)|^2\,d\eta\le C.
$$
Since $r_{2m+1}\ge 0$ and $\lim_{t\rightarrow 0}r_{2m+1}(t|\eta|)=1/2^{2j}(2m+1)!$ pointwise in $\eta$, an application of Fatou's lemma yields that $|\eta|^{2m}|\widehat V(\eta)|^2$ is integrable.

If $d\ge 4$ is even, we use \eqref{eqn:trace2even} for $a_{2,V}(t)$.
Dominated convergence implies $a_{2,V}\in C^0\bigl(\Rbarplus\bigr)$. Furthermore $|a_{2,V}(t)|\le \|V\|_{L^2}^2/(2\pi)^{\frac d2}$, with equality at $t=0$.

If $V\in H^m(\R^d)$, differentiating under the integral sign and evaluating the beta integral shows that $a_{2,V}\in C^{2m}(\overline{\R_+})$, with
\begin{equation*}
\begin{split}
|\partial_t^j a_{2,V}(t)|&\le \frac{\Gamma(\frac j2+1)^2}{(2\pi)^{\frac d2}\Gamma(j+2)}
\bigl\||D|^{\frac j2}V\|_{L^2}^2\\
&\le \frac 1{2^j(2\pi)^{\frac d2}}\bigl\||D|^{\frac j2}V\|_{L^2}^2.
\end{split}
\end{equation*}
We note that the same bounds hold for $d=2$.

To control the remainder term in the Taylor expansion, we write
$$
\cos(x)=\sum_{j=0}^{m-1}\frac{(-1)^jx^{2j}}{(2j)!}+(-1)^mr_{2m}(x^2)x^{2m}
$$
where $r_{2m}\ge 0$, and
$$
r_{2m}(0)=\frac{1}{(2m)!}\,,\qquad \sup_{s\in\R}\,r_{2m}(s)\le\frac{1}{(2m)!}.
$$
Then
\begin{multline*}
\int_0^1 \cos\Bigl(t|\eta|\sqrt{s-s^2\,}\,\Bigr)\,ds\\=
\sum_{j=0}^{m-1} \frac{(-1)^j(j!)^2}{(2j)!(2j+1)!}\,t^{2j}|\eta|^{2j}+
(-1)^m\tilde r_{2m}(t^2|\eta|^2)\,t^{2m}|\eta|^{2m},
\end{multline*}
where $\tilde r_{2m}$ is the nonnegative function defined by
$$
\tilde r_{2m}(t^2|\eta|^2)=2\int_0^1 s^m(1-s)^mr_{2m}\bigl(s(1-s)t^2|\eta|^2\bigr)\,ds.
$$
We conclude that
$$
c_{2,j}=\frac{(-1)^{j+\frac d2}(j!)^2\,\||D|^jV\|_{L^2}^2}{(2\pi)^{\frac d2}(2j)!(2j+1)!}\,.
$$
Since $\tilde r_{2m}(0)\ge 0$, the remainder of the proof follows as for $d$ odd.
\end{proof}


\section{Proof of Theorem \ref{thm:traceWk} when $d\ge 4$.} 

For $d\ge 4$ we will deduce Theorem \ref{thm:traceWk} from an expansion of the multiplier that determines $a_{k,V}(t)$. Recall that $Q_{k,\bfs}(\eta')=Q_{k,\bfs}(0,\eta_2,\ldots,\eta_k)$, where $Q_{k,\bfs}$ is defined in \eqref{eqn:etaform}. In matrix form, $Q_{k,\bfs}$ takes the form
$$
Q_{k,\bfs}(\eta')=[0,\eta_2,\ldots,\eta_k]\cdot[\text{diag}(\bfs)-\bfs\otimes\bfs]\cdot[0,\eta_2,\ldots,\eta_k]^T,
$$
where $\eta_i\cdot\eta_j$ is the dot product, and $\bfs$ acts by multiplying $\eta_j$ by $s_j$.
We express this in terms of the variables of $\widehat V$. For $1\le j\le k$ let
$$
\theta_j=\eta_{j+1}-\eta_j,\quad\text{where}\;\eta_1=\eta_{k+1}=0.
$$
Then
$$
\eta_i=\sum_{j<i}\theta_j=-\sum_{j\ge i}\theta_j.
$$
We use these representations in the quadratic expression for $Q_{k,\bfs}(\eta')$, according to the upper or lower diagonal parts of $Q_{k,\bfs}(\eta')$, to see that
\begin{equation*}
Q_{k,\bfs}(\eta')=\sum_{i<j}q_{k,\bfs}(i,j)\la\theta_i,\theta_j\ra,
\end{equation*}
where
$$
q_{k,\bfs}(i,j)=\sum_{i<l\le j} (s_\ell-s_\ell^2)-2\sum_{i<\ell<m\le j}s_\ell s_m.
$$
In particular, $|q_{k,\bfs}(i,j)|\le 1$. Therefore,
$$
|Q_{k,\bfs}(\eta')|\le \sum_{i< j}|\theta_i||\theta_j|,\qquad
|Q_{k,\bfs}(\eta')|^{\hf}\le \sum_i |\theta_i|.
$$

We thus conclude from Lemma \ref{lem:Vhatprod} that
\begin{multline}\label{eqn:Qkbound}
\int_{\Delta^{k-1}}\int_{\R^{d(k-1)}}Q_{k,\bfs}(\eta')^{\frac j2}|\widehat{V}(\eta_2)||\widehat{V}(\eta_3-\eta_2)|\cdots|\widehat{V}(-\eta_k)|\,d\eta'\,d\bfs
\\ 
\le
\frac{k^j(2\pi)^{d(k-1)}}{(k-1)!}\|V\|_{\hatL}^{k-2}\|V\|_{H^{\lceil j/2 \rceil}}^2.
\end{multline}

Consider first the case where $2k\ge d$. We use \eqref{eqn:akVform} to express $a_{k,V}(t)$.
Differentiation under the integral sign, together with \eqref{eqn:Gbound} and \eqref{eqn:Qkbound}, shows that if $V\in \hatL\cap H^m$, then $a_{k,V}\in C^{2m}(\overline{\R_+})$, and for $0\le j\le 2m$,
$$
|\partial_t^j a_{k,V}(t)|\le \frac{2^{d+3-2k}\pi^{\frac d2}k^j\Gamma(\frac{j+1}2)}{(2\pi)^d k!\Gamma(\frac{2k-d+j+1}2)}
\|V\|_{\hatL}^{k-2}\|V\|_{H^{\lceil j/2 \rceil}}^2.
$$
We refer to \eqref{eqn:tildeakeven}-\eqref{eqn:tildeakodd}, and use Lemma \ref{lem:tildeak} to conclude that
\begin{align}\notag
|\partial_t^j \tilde a_{k,V}(t)|&\le 
\frac{2\,k^j\,\Gamma\bigl(\frac{j+1}2\bigr)}
{(2\pi)^{\frac d2}2^{2(k-2)}k!\Gamma\bigl(\frac{j+1}2+k-2\bigr)}\|V\|_{\hatL}^{k-2}\|V\|_{H^{\lceil j/2 \rceil}}^2\\
\notag
&\le \frac{2\,k^j\,\Gamma\bigl(\frac 12\bigr)}{(2\pi)^{\frac d2}2^{2(k-2)}k!\,\Gamma\bigl(k-\frac 32\bigr)}\|V\|_{\hatL}^{k-2}\|V\|_{H^{\lceil j/2 \rceil}}^2\\
&\le \frac{C_d \, k^{j-2}}{(2k-4)!}\|V\|_{\hatL}^{k-2}\|V\|_{H^{\lceil j/2 \rceil}}^2.
\label{eqn:tildeakbound}
\end{align}

If $2k<d$, then we express $\trace\la\phi,W_{k,V}\ra$ in the form \eqref{eqn:traceWk'}, where
$a_{k,V}(t)$ is given by \eqref{eqn:tracekeven} or \eqref{eqn:tracekodd}, respectively if $d$ is even or odd. Lemma \ref{lem:tildeak} leads to the bound \eqref{eqn:tildeakbound} in this case as well.

We conclude that, when $V\in \hatL\cap H^m$, then for $j\le 2m$,
\begin{equation}\label{eqn:alphabound}
|\partial_t^j\alpha_V(t)|\le C_d\,\|V\|_{\lceil j/2 \rceil}^2\,
p_j\Bigl(t\ts \|V\|_{\hatL}^{1/2}\Bigr)
\cosh\Bigl(t\ts \|V\|_{\hatL}^{1/2}\Bigr),
\end{equation}
where $p_j$ is a polynomial of order at most $\max(0,j-2)$.

The Taylor polynomial for $a_{k,V}(t)$ of order $2m$ at $t=0$ is of the form
$$
a_{k,V}(t)=\sum_{j=0}^m c_{k,j}\,t^{2j}\int_{\Delta^{k-1}}\int_{\R^{d(k-1)}}Q_{k,\bfs}(\eta')^j\widehat{V}(\eta_2)\cdots\widehat{V}(-\eta_k)\,d\eta'\,d\bfs,
$$
where the coefficients $c_{k,j}$ can be read off from \eqref{eqn:akVform} and \eqref{eqn:Gseries}. The Taylor polynomial of $\tilde a_{k,V}$ can then be deduced from \eqref{eqn:tildeakeven} and \eqref{eqn:tildeakodd}.


\section{Proof of Theorem \ref{thm:traceWk} when $d\le 3$.} 
To show that $\trace\la \phi,W_{k,V}\ra$ admits an expansion for $V\in L^\infty_{\comp}(\R^d)$ if $d\le 3$, we establish a representation that involves a multilinear integral of $V$ instead of $\widehat{V}$. Throughout this section we assume that $V\in L^\infty_{\comp}(\R^d)$.
\begin{theorem}\label{thm:spatial}
If $d=1,2,3$, and $k\ge 2$, then one can write
$$
\trace\la \phi,W_{k,V}\ra=\int_0^\infty t^{2k-d}\phi(t)\int_{\R^{dk}}\biggl(\,V(u_1)\prod_{j=2}^k V(u_1+tu_j)\biggr)\,d\sigma_k(u')\,du_1\,dt,
$$
with $d\sigma_k(u')$ a finite positive measure on $\R^{d(k-1)}$, supported in the set
$$
|u_2|+|u_3-u_2|+\cdots +|u_k-u_{k-1}|+|u_k|\le 1.
$$
Furthermore,
\begin{equation}\label{eqn:intdsigma}
\int d\sigma_k(u')=
\begin{cases} \Bigl(2^{2k-3}k!(k-1)!\Bigr)^{-1},& d=1,\\
2\Bigl(\pi k(2k-2)!\Bigr)^{-1},& d=2,\\
\Bigl(2^{2k-3}\pi k!(k-2)!\Bigr)^{-1},& d=3.
\end{cases}
\end{equation}
If $d=1,2$, then $d\sigma_k(u')$ is absolutely continuous with respect to $du'$.
\end{theorem}

\begin{corollary}\label{cor:holderest}
For $d=1,2,3$, and $k\ge 2$, one can write
$$
\trace\la\phi,W_{k,V}\ra=\int_0^\infty t^{2k-d}a_{k,V}(t)\,\phi(t)\,dt,
$$
where $a_{k,V}(t)\in C(\overline{\R_+})$, and satisfies the following bound when $1\le p_j\le \infty$ and $\sum_{j=1}^k p_j^{-1}=1$, 
$$
|a_{k,V}(t)|\le 
\biggl(\int d\sigma_k(u')\biggr)\prod_{j=1}^k \|V(x_j)\|_{p_j}.
$$
\end{corollary}
\begin{proof}
By Theorem \ref{thm:spatial}, equality holds with
$$
a_{k,V}(t)=\int_{\R^{3k}}\Bigl(\,V(u_1)\prod_{j=2}^k V(u_1+tu_j)\Bigr)\,du_1\,d\sigma_k(u').
$$
We apply H\"older's inequality to estimate the integral over $u_1$. Continuity follows by continuity of translation in $L^p$ for $p<\infty$ and the compact support of $d\sigma_k$.
\end{proof}

We remark that for $1\le d\le 3$, Corollary \ref{cor:holderest} and \eqref{eqn:intdsigma} imply
$$
|a_{k,V}(t)|\le \frac{\|V\|_{L^\infty}^{k-2}\|V\|_{L^2}^2}{(2k-2)!}.
$$
and hence
\begin{equation*}
\begin{split}
|\alpha_V(t)|&=\biggl|\sum_{k=2}^\infty (-1)^k t^{2(k-2)}a_{k,V}(t)\biggr|\\
&\le \|V\|_{L^2}^2\,\cosh\Bigl(t\,\|V\|_{L^\infty}^{1/2}\Bigr).
\end{split}
\end{equation*}

To prove Theorem \ref{thm:spatial}, we recall from Section \ref{sec:traceform} that
$$
\trace\la\phi,W_{k,V}\ra=\lim_{\epsilon\rightarrow 0}\trace\biggl(
\frac 4k\int_{\R_+^k}\psi(s_1+\cdots +s_k)\Sin_\epsilon(s_k)V\cdots V 
\Sin_\epsilon(s_1)V\,d^k\! s\biggr).
$$
For $1\le d \le 3$, the operator $\Sin_\epsilon(s)$ is convolution with respect to the measure $\rho_\epsilon*\Sin(s,x)$, where $\rho_\epsilon(x)=\epsilon^{-d}\rho(\epsilon^{-1}x)$, and
\begin{equation*}
\Sin(s,x)=\begin{cases}
\frac 12\one_{[-s,s]}(x),& d=1,\\
(2\pi)^{-1}\bigl(s^2-|x|^2\bigr)_+^{-1/2}, & d=2,\rule{0pt}{16pt}\\
(4\pi |x|)^{-1}\delta(s-|x|),& d=3.\rule{0pt}{16pt}
\end{cases}
\end{equation*}
The kernel of $\Sin_\epsilon(s,x)$ is smooth and compactly supported, and for $V\in L^\infty_{\comp}$ the trace is given by integrating the kernel of the composition over the diagonal. This leads to the following formula for $\trace\la\phi,W_{k,V}\ra$,
$$
\lim_{\epsilon\rightarrow 0}\frac 4k\int_{\R^{dk}}\int_{\R_+^k}\psi(s_1+\cdots +s_k)\prod_{j=1}^k\Sin_\epsilon(s_j,x_{j+1}-x_j)V(x_j)\,d^k\!s\,dx.
$$
We proceed to analyze the limit, starting with the case $d=3$.

\subsection{The case $d=3$}
\begin{lemma}
Let $d=3$ and $\psi(s)=s\phi(s)$, and assume $V\in L^\infty_{\comp}(\R^3)$. Then $\trace\la \phi, W_{k,V}\ra$ is equal to
\begin{equation}\label{eqn:ktraceform'}
\frac{4}{k(4\pi)^k}\int_{\R^{3k}}\psi\Bigl(|x_2-x_1|+\cdots +|x_1-x_k|\Bigr)\prod_{j=1}^k
\,\frac{V(x_j)}{|x_{j+1}-x_j|}\;dx.
\end{equation}
\end{lemma}
\begin{proof}
Let $d\Gamma(s,y)=(4\pi |y|)^{-1}\delta(s-|y|)$ denote $2^{-\hf}|y|$ times surface measure on the forward light cone $|y|=s$. Then by the above, we can write $\trace\la \phi,W_{k,V}\ra$ as  $4k^{-1}$ times the following
$$
\lim_{\epsilon\rightarrow 0}
\int\psi(s_1+\cdots +s_k)\prod_{j=1}^k\,\rho_\epsilon(x_{j+1}-x_j-y_j)\,V(x_j)\,d\Gamma(s_j,y_j)\,dy\,d^k\!s\,dx.
$$
The integrand is bounded and of compact support, so we may integrate first over $s$ to see this equals
$$
\frac{4}{k(4\pi)^k}\lim_{\epsilon\rightarrow 0}
\int\frac{\psi\bigl(|y_1|+\cdots +|y_k|\bigr)}{|y_1|\cdots|y_k|}\biggl(\prod_{j=1}^k\,\rho_\epsilon(x_{j+1}-x_j-y_j)\,V(x_j)\biggr)\,dy\,dx.
$$
The bound $|\psi(s)|\le s$ and compact support of $V$ shows that the integral in $x,y$ over the region $\min_j |y_j|<\delta$ goes to $0$ as $\delta\rightarrow 0$, uniformly in $\epsilon$.
Uniform convergence away from this set for the convolution over $y$ then shows that the integral converges to \eqref{eqn:ktraceform'}.
\end{proof}

We introduce variables $x_1=u_1$, and $x_j=u_1+u_j$ for $2\le j\le k$, to write the trace as
\begin{equation*}
\frac{4}{k(4\pi)^k}\int_{\R^{3k}}\psi(f(u'))
\,\frac{V(u_1)V(u_1+u_2)\cdots V(u_1+u_k)}{|u_2||u_3-u_2|\cdots|u_k-u_{k-1}||u_k|}\;du'\,du_1,
\end{equation*}
where
$$
f(u')=|u_2|+|u_3-u_2|\cdots +|u_k-u_{k-1}|+|u_k|.
$$

The function $f$ is smooth on the open subset of $\R^{3(k-1)}$ where $u_{j+1}\ne u_j$, and $u_2\ne 0$, $u_k\ne 0$. Furthermore, with the convention $u_1=u_{k+1}=0$,
$$
\nabla_{u_j}f(u')=\frac{u_j-u_{j-1}}{|u_j-u_{j-1}|}-\frac{u_{j+1}-u_j}{|u_{j+1}-u_j|}\,,
$$
which vanishes exactly when $u_j$ lies in the convex hull $[u_{j-1},u_{j+1}]$. Let $\Omega$ be the open subset of $\R^{3(k-1)}$ where $u_2,u_k\ne 0$, and no three $u_j$'s are co-linear. 
Then $f$ has non-zero gradient with respect to each $u_j$ at all points in $\Omega$, and since the complement of $\Omega$ is measure $0$ we may write
$$
\trace\la \phi,W_{k,V}\ra=\int_0^\infty \phi(t)\int_{\R^{3k}}\biggl(\,V(u_1)\prod_{j=2}^k V(u_1+u_j)\biggr)\,d\sigma_{k,t}(u')\,du_1\,dt,
$$
where $d\sigma_{k,t}(u')$ is the following positive measure,
$$
d\sigma_{k,t}(u')=\frac{4}{k(4\pi)^k}\,\frac{ t\ts \delta\bigl(t-f(u')\bigr)\,\one_\Omega(u')}{|u_2||u_3-u_2|\cdots|u_k-u_{k-1}||u_k|}
\;du'.
$$
If we denote $d\sigma_k(u')=d\sigma_{k,1}(u')$, then by changing $u'\rightarrow tu'$, we obtain
$$
\trace\la \phi,W_{k,V}\ra=\int_0^\infty \phi(t)\,t^{2k-3}\int_{\R^{3k}}\biggl(\,V(u_1)\prod_{j=2}^k V(u_1+tu_j)\biggr)\,d\sigma_k(u')\,du_1\,dt.
$$

The above steps show more generally that
\begin{multline*}
\trace\biggl(
\frac 4k\int_{\R_+^k}\psi(s_1+\cdots +s_k)\Sin(s_k)V_k\cdots V_2\,
\Sin(s_1)V_1\,d^k\!s\biggr)\\
=\int_0^\infty \phi(t)\,t^{2k-3}\int_{\R^{3k}}\biggl(\,V_1(u_1)\prod_{j=2}^k V_j(u_1+tu_j)\biggr)\,d\sigma_k(u')\,du_1\,dt.
\end{multline*}
The proof of Lemma \ref{lem:akVform} shows that when $d\le 2k$, this equals
\begin{multline*}
\frac{4\pi^{\frac d2}}{k(2\pi)^{dk}}\int_0^\infty \phi(t)\, t^{2k-d}
\int_{\R^{d(k-1)}}\int_{\Delta^{k-1}}
G_{k-\frac {d+1}2}\Bigl(t^2Q_{k,\bfs}(\eta')\Bigr)\times\\
\widehat{V_2}(\eta_2)\widehat{V_3}(\eta_3-\eta_2)\cdots\widehat{V_k}(\eta_k-\eta_{k-1})\widehat{V_1}(-\eta_k)\,d\bfs\,d\eta'\,dt.
\end{multline*}
This yields the relation, with $d=3$ in this case, and $k\ge 2$,
\begin{equation}\label{eqn:ftsigma}
\widehat{d\sigma_k}(-\eta')=\frac{4\pi^{\frac d2}}{k(2\pi)^d}
\int_{\Delta^{k-1}}
G_{k-\frac {d+1}2}\Bigl(Q_{k,\bfs}(\tilde{\eta})\Bigr)\,d\bfs,
\end{equation}
where
$$
\tilde{\eta}=(0,\eta_2,\eta_2+\eta_3,\ldots,\eta_2+\cdots+\eta_{k-1}+\eta_k).
$$
By Lemma \ref{lem:Gkzero}, for $d=3$ we deduce that
\begin{equation*}
\int_{\R^{3(k-1)}} d\sigma_k(u')=\frac{1}{ 2^{2k-3}\ts\pi\ts k!(k-2)!}\,.
\end{equation*}


\subsection{The case $d=1$}
A similar analysis, using the formula in $d=1$:
$$
\Sin(s;x,y)=\frac 12\one_{[-s,s]}(x-y)
$$
leads to the following formula,
\begin{equation*}
\trace\la \phi,W_{k,V}\ra
=\int_0^\infty \phi(t)\,t^{2k-1}\int_{\R^k} \Bigl(V(u_1)\prod_{j=2}^k V(u_1+tu_j)\Bigr)\,\rho_k(u')\,du'\,du_1,
\end{equation*}
where
$$
\rho_k(u')=\frac 1{2^{k-2}k!}\bigl(1-f(u')\bigr)_+^{k-1}.
$$
By \eqref{eqn:ftsigma} and Lemma \ref{lem:Gkzero}, we have
$$
\int_{\R^{k-1}} \rho_k(u')\,du'=\frac{1}{2^{2k-3}k!(k-1)!}.
$$

\subsection{The case $d=2$}
If $d=2$ we do not have a closed form for $d\sigma_k(u')$, but observe that it is an integrable function for $k\ge 2$. For this, we note that $d\sigma_k(u')=\rho_k(u')\,du'$, where $\rho_k(u')$ equals $4/k(2\pi)^k$ times the following,
$$
\int_{\Delta^{k-1}}
\bigl(s_1^2-|u_2|^2\bigr)_+^{-1/2}\bigl(s_2^2-|u_3-u_2|^2\bigr)_+^{-1/2}\cdots \bigl(s_k^2-|u_k|^2\bigr)_+^{-1/2}\,d\bfs.
$$
For $k=2$, one has the estimate for $|u_2|\le\hf$,
$$
\int_{|u_2|}^{1-|u_2|}\bigl(s^2-|u_2|^2\bigr)^{-\hf}\bigl((1-s)^2-|u_2|^2\bigr)^{-\hf}\,ds
\approx 1+|\log(1-2|u_2|)|.
$$
We next observe that
$$
\Bigl\|\bigl(s^2-|u|^2\bigr)_+^{-1/2}\Bigr\|_{L^p(\R^2,du)}=C_p\,s^{\frac 2p-1},\quad 1\le p <2.
$$
When $k\ge 3$ we take $p=k/(k-1)$, and use the H\"older and Young inequalities to see that
\begin{multline*}
\int_{\R^{2(k-1)}}\bigl(s_1^2-|u_2|^2\bigr)_+^{-1/2}\bigl(s_2^2-|u_3-u_2|^2\bigr)_+^{-1/2}\cdots \bigl(s_k^2-|u_k|^2\bigr)_+^{-1/2}\,du'\\
\le C_k\prod_{j=1}^k s_j^{1-\frac 2k}.
\end{multline*}
Tonelli's theorem now implies integrability of $\rho_k(u')$.  Lemma \ref{lem:Gkzero} and \eqref{eqn:ftsigma} yield
$$
\int \rho_k(u')\,du'=\frac{2}{\pi k(2k-2)!}\,.
$$

\begin{proof}[Proof of Theorem \ref{thm:traceWk} for $d\le 3$.]
The first part of Theorem \ref{thm:traceWk} follows from Corollary \ref{cor:holderest}, so it remains to establish differentiability of $a_{k,V}$. 
We first consider the case $m=1$, and show that $a_{k,V}(t)\in C^2(\overline{\R_+})$, with bounds
\begin{equation}\label{eqn:dt2akbound}
\begin{split}
|\partial_t a_{k,V}(t)|&\le (k-1)\biggl(\int d\sigma_k(u')\biggr)\|\nabla V\|_{L^2}\|V\|_{L^2}\|V\|_{L^\infty}^{k-2},\\
|\partial_t^2 a_{k,V}(t)|&\le (k-1)^2\biggl(\int d\sigma_k(u')\biggr)\|\nabla V\|_{L^2}^2\|V\|_{L^\infty}^{k-2}.
\end{split}
\end{equation}
We present the details for $d=3$. Assume $V\in L^\infty_{\comp}\cap H^1(\R^3)$, and consider
$$
a_{k,V}(t)=\int
V(u_1)\,\biggl(\,\prod_{j=2}^kV(u_1+tu_j)\biggr)\,du_1\,d\sigma_k(u').
$$
if we apply $\partial_t$ to $a_{k,V}(t)$, we formally obtain the following 
\begin{align*}
\sum_{j=2}^k
\int
V(u_1)\biggl(\,\prod_{i\ne 1,j} V(u_1+tu_i)\biggr)\bigl(\la u_j,\nabla\ra V\bigr)(u_1+tu_j)\,du_1\,d\sigma_k(u').
\end{align*}
That this equals $\partial_t a_{k,V}(t)$ can be proven rigorously by taking difference quotients, using H\"older's inequality as in Corollary \ref{cor:holderest}, continuity of translation in $L^p$ for $p<\infty$,
and the following, which holds for $V\in H^1(\R^d)$,
\begin{equation}\label{eqn:diffquot}
\lim_{h\rightarrow 0}\sup_{|u_j|\le 1}
\biggl\|\frac{V(u_1+hu_j)-V(u_1)}{h}-\bigl(\la u_j, \nabla\ra V\bigr)(u_1)\biggr\|_{L^2(du_1)}=0.
\end{equation}

Setting $u_1=tx_1$, $u_i=x_i-x_1$ for $i>1$, then expresses $\partial_t a_{k,V}(t)$ as
\begin{multline*}
c_k\sum_{j=1}^k t^3\int_{\R^{dk}}\frac{\Bigl(\,\prod_{i\ne j} V(tx_i)\Bigr)\bigl(\la x_j-x_1,\nabla\ra V\bigr)(tx_j)}{|x_2-x_1|\cdots |x_k-x_{k-1}||x_1-x_k|}\\
\times\delta\Bigl(1-\sum_i |x_{i+1}-x_i|\Bigr)\,dx.
\end{multline*}
We make a cyclic relabeling of the indices that takes $j$ to $1$, and reverse the above change of variables to see that 
$\partial_t a_{k,V}(t)$ is equal to
$$
-c_k\sum_{j=2}^k \int\bigl(\la u_j,\nabla\ra V\bigr)(u_1)\biggl(\,\prod_{i=2}^k V(u_1+tu_i)\biggr)du_1\,d\sigma_k(u').
$$
Differentiation in $t$ now leads to the formula
\begin{multline*}
\partial_t^2 a_{k,V}(t)=c_k
\sum_{i,j=2}^k\int\bigl(\la u_j,\nabla\ra V\bigr)(u_1)
\bigl(\la u_i,\nabla\ra V\bigr)(u_1+tu_i)\\
\times
\biggl(\,\prod_{m\ne 1,i} V(u_1+tu_m)\biggr)du_1\,d\sigma_k(u').
\end{multline*}
Since $|u_j|\le 1$ on the support of $d\sigma_k(u')$, this implies the bound \eqref{eqn:dt2akbound}.

To prove that $\partial_t^2 a_{k,V}$ is continuous requires more care than for $\partial_t a_{k,V}$, since translation is not strongly continuous in $L^\infty(\R^d)$. One instead uses Lusin's Theorem applied to $V$, and absolute continuity of $|\nabla V|^2$, to see that the integral over $u_1$ of the integrand for $\partial_t^2 a_{k,V}$ is continuous in $(t,u')$.
Similar steps together with \eqref{eqn:diffquot} justify convergence of the difference quotients of $\partial_t a_{k,V}$ to the above formula for $\partial_t^2 a_{k,V}$.

For $m\ge 2$, we can apply the above proof and induction to see that $a_{k,V}\in C^{2m}(\overline{\R_+})$. Additionally, $\partial_t^j a_{k,V}(t)$ can be written as a sum of at most
$(k-1)^j$ terms, each of the form
$$
\int\bigl(\la u,\nabla\ra^{j_1} V\bigr)(u_1)
\biggl(\,\prod_{m=2}^k \bigl(\la u,\nabla\ra^{j_m} V\bigr)(u_1+tu_m)\biggr)du_1\,d\sigma_k(u'),
$$
where $\la u,\nabla\ra$ indicates some $\la u_i,\nabla\ra$, possibly different in each occurrence, and
$$
\max j_i\le \lceil j/2\rceil,\qquad \sum_{i=1}^k j_i=j.
$$
The Gagliardo--Nirenberg inequalities, see \cite[Lemma 3.4]{SmZw} which follows from \cite[(3.17) in \S 13.3]{pde}, then bounds
\begin{equation*}
|\partial_t^j a_{k,V}(t)|\le C_{d,j}(k-1)^j\biggl(\int d\sigma_k(u')\biggr)
\|V\|_{L^\infty}^{k-2}\|V\|_{\lceil j/2\rceil}^2.
\end{equation*}
In particular, the identities \eqref{eqn:intdsigma} imply
\begin{equation}\label{eqn:alphabound3}
|\partial_t^j a_{k,V}(t)|\le
\frac{C_{d,j}(k-1)^j}{(2k-2)!}\|V\|_{L^\infty}^{k-2}\|V\|_{\lceil j/2\rceil}^2.
\end{equation}
We may then sum over $k$ to obtain the bound \eqref{eqn:alphajbound} with $X_d=L^\infty_{\comp}$.
\end{proof}


\bibliographystyle{plain}
\bibliography{wave_trace}


\end{document}